\documentclass[11 pt]{article}
\title{A Birman exact sequence for $\Aut(F_n)$}
\author{Matthew Day\footnote{Supported in part by an NSF postdoctoral fellowship}\ \ and Andrew Putman\footnote{Supported in part by NSF grant DMS-1005318}}
\usepackage{amsmath}
\usepackage{amscd}
\usepackage[vscale=0.76]{geometry}
\usepackage{amssymb}
\usepackage[amsmath,amsthm,thmmarks]{ntheorem}
\usepackage{amsfonts}
\usepackage{calc}
\usepackage[font=small]{caption}
\usepackage{mathabx}
\usepackage{accents}
\usepackage{booktabs}
\usepackage{epsfig}
\usepackage{verbatim}

\newtheorem{theorem}{Theorem}[section]
\newtheorem{maintheorem}{Theorem}
\newtheorem{proposition}[theorem]{Proposition}
\newtheorem{lemma}[theorem]{Lemma}

\newtheorem{corollary}[theorem]{Corollary}

\newtheorem{case}{Case}

\newcommand\BeginCases{\setcounter{case}{0}}

\theoremstyle{definition}

\theoremstyle{remark}

\newtheorem*{remark}{Remark}

\DeclareMathOperator{\Hom}{Hom}
\DeclareMathOperator{\CHom}{Der}

\DeclareMathOperator{\Ker}{Ker}

\DeclareMathOperator{\Mod}{Mod}

\DeclareMathOperator{\IA}{IA}

\DeclareMathOperator{\GL}{GL}

\newcommand\Z{\ensuremath{\mathbb{Z}}}
\newcommand\Q{\ensuremath{\mathbb{Q}}}

\DeclareMathOperator{\HH}{H}
\DeclareMathOperator{\CC}{C}

\DeclareMathOperator{\supp}{supp}
\DeclareMathOperator{\mult}{mult}
\DeclareMathOperator{\Ab}{ab}

\DeclareMathOperator{\End}{End}
\DeclareMathOperator{\Aut}{Aut}
\DeclareMathOperator{\Out}{Out}

\newcommand\Span[1]{\ensuremath{\langle #1 \rangle}}

\newcommand\Set[2]{\ensuremath{\{\text{#1 $|$ #2}\}}}

\newcommand\AutFB[3]{\ensuremath{\mathcal{A}_{#1,#2,#3}}}
\newcommand\BKer[3]{\ensuremath{\mathcal{K}_{#1,#2,#3}}}
\newcommand\BKernkl{\ensuremath{\mathcal{K}}}

\newcommand\LKernkl{\ensuremath{\mathcal{L}}}
\newcommand\GroupPres[2]{\ensuremath{\langle \text{$#1$ $|$ $#2$} \rangle}}
\newcommand\GroupLPres[4]{\ensuremath{\langle \text{$#1$ $|$ $#2$ $|$ $#3$ $|$ $#4$} \rangle}}

\DeclareMathOperator{\conj}{conj}
\newcommand\Conj[1]{\ensuremath{\ldbrack #1 \rdbrack}}
\newcommand\Mul[2]{\ensuremath{M_{#1,#2}}}
\newcommand\Con[2]{\ensuremath{C_{#1,#2}}}

\newcommand\Id[1]{\ensuremath{\mathrm{id}_{#1}}}
\newcommand{\Vc}{V_{C}}
\newcommand{\BVc}{\overline{V}_{C}}
\newcommand{\Vcy}{V_{Y}}
\newcommand{\BVcy}{\overline{V}_{Y}}

\newcommand{\co}{\colon\,}
\newcommand{\WedgeTwo}{\ensuremath{{\bigwedge}^{\hspace{-3pt}2}}}

\begin{document}

\maketitle

\begin{abstract}
The Birman exact sequence describes the effect on the mapping class group of a surface with boundary of
gluing discs to the boundary components.  We construct an analogous exact sequence for the automorphism
group of a free group.  For the mapping class group, the kernel of the Birman exact sequence is a surface braid group.  We prove
that in the context of the automorphism group of a free group, the natural kernel is finitely generated.
However, it is not finitely presentable; indeed, we prove that its second rational homology group has 
infinite rank by constructing an explicit infinite collection of linearly independent abelian cycles.  
We also determine the abelianization of our kernel and build a simple infinite presentation for it.  
The key to many of our proofs are several new generalizations of the Johnson homomorphisms.
\end{abstract}

\section{Introduction}
\label{section:introduction}

When studying the mapping class group $\Mod(\Sigma)$ of a closed orientable surface $\Sigma$, one
is led inexorably to the mapping class groups of surfaces with boundary.  For instance, often
phenomena are ``concentrated'' on a subsurface of $\Sigma$, and thus they ``live'' in the mapping
class group of the subsurface (which is a surface with boundary).  The key tool for understanding
the mapping class group of a surface with boundary is the {\em Birman exact sequence}.  If
$S$ is a surface with $p$ boundary components and $\hat{S}$ is the closed surface that results
from gluing discs to the boundary components of $S$, then there is a natural surjection
$\psi \co \Mod(S) \rightarrow \Mod(\hat{S})$.  Namely, if $f \in \Mod(S)$, then
$\psi(f)$ is obtained by extending $f$ over the glued-in discs.  Excluding some
degenerate low-genus cases, a basic result of Birman \cite{BirmanSequence} shows that
$\Ker(\psi)$ is isomorphic to the $p$-strand braid group $B_p(\hat{S})$ on $\hat{S}$
(or possibly a slight modification of $B_p(\hat{S})$ depending
on your conventions for mapping class groups of surfaces with boundary).
This is summarized in the exact sequence
\begin{equation}
\label{eqn:birmanmod}
1 \longrightarrow B_p(\hat{S}) \longrightarrow \Mod(S) \longrightarrow \Mod(\hat{S}) \longrightarrow 1.
\end{equation}
The purpose of this paper is to develop a similar exact sequence for the automorphism group of a free group.

\paragraph{Automorphisms of free groups with boundary, motivation.}
We begin by giving a definition of the ``automorphism group of a free group with boundary''.
A motivation for this definition is as follows.  A prototypical example of where
mapping class groups of surfaces with boundary arise is as the stabilizer of a 
collection of homotopy classes of non-nullhomotopic disjoint simple closed
curves $\gamma_1,\ldots,\gamma_k$ on a genus $g$ surface $\Sigma_g$ 
(such a collection forms a simplex in the {\em curve complex}).  The stabilizer of the $\gamma_i$ is then
essentially the mapping class group of the result of cutting $\Sigma_g$ along the $\gamma_i$.  We
wish to give an algebraic description of such collections. 

Recall that there is a bijection between conjugacy classes
in $\pi_1(\Sigma_g)$ and homotopy classes of oriented closed curves on $\Sigma_g$.  We want
to understand which collections of conjugacy classes in $\pi_1(\Sigma_g)$ correspond
to collections $\gamma_1,\ldots,\gamma_k$ 
of homotopy classes of disjoint non-nullhomotopic simple closed curves
on $\Sigma_g$ (oriented in some way).  The description is simplest if we restrict to collections of $\gamma_i$
such that $\gamma_1 \cup \cdots \cup \gamma_k$ does not separate $\Sigma_g$.
Elements $a_1,b_1,\ldots,a_g,b_g \in \pi_1(\Sigma_g)$ form a
{\em standard basis} for $\pi_1(\Sigma_g)$ if they generate $\pi_1(\Sigma_g)$ and
satisfy the surface relation $[a_1,b_1]\cdots[a_g,b_g]=1$.  We then have the following folklore fact.
As notation, if $G$ is a group and $g \in G$, then $\Conj{g}$ will denote
the conjugacy class of $g$.

\begin{proposition}
\label{proposition:curvesdesc}
A collection $c_1,\ldots,c_k$ of conjugacy classes in $\pi_1(\Sigma_g)$ corresponds
to a collection $\gamma_1,\ldots,\gamma_k$ of homotopy classes of disjoint oriented non-nullhomotopic
simple closed curves on $\Sigma_g$ such that $\Sigma_g \setminus (\gamma_1 \cup \cdots \cup \gamma_k)$
is connected if and only if there exists a standard basis $a_1,b_1,\ldots,a_g,b_g$ for $\pi_1(\Sigma_g)$
with $c_i = \Conj{a_i}$ for $1 \leq i \leq k$.
\end{proposition}

\noindent
We do not know a reference for Proposition~\ref{proposition:curvesdesc}, so we include a proof in Appendix \ref{appendix:curvesdesc}.

Proposition~\ref{proposition:curvesdesc} suggests defining the automorphism group of a free group with boundary to be the
subgroup of the automorphism group of a free group fixing the conjugacy classes of some fixed partial basis
for the free group.

\paragraph{Automorphisms of free groups with boundary, definition.}
Let $F_{n,k,l}$ be the free group on letters $\{x_1,\ldots,x_n,y_1,\ldots,y_k,z_1,\ldots,z_l\}$ and set
\[X = \{x_1,\ldots,x_n\} \quad \text{and} \quad Y = \{y_1,\ldots,y_k\} \quad \text{and} \quad Z = \{z_1,\ldots,z_l\}.\]
We define
\[\AutFB{n}{k}{l} = \Set{$f \in \Aut(F_{n,k,l})$}{$\Conj{f(v)} = \Conj{v}$ for $v \in Y \cup Z$}.\]
The reason for separating the roles of the $y_i$ and the $z_j$ will become apparent shortly.  
The group $\AutFB{n}{k}{l}$ should be viewed as the automorphism group of a free group on $n$ 
letters with $k+l$ boundary components.  

\begin{remark}
There have been many proposals for analogues of the curve complex for $\Aut(F_n)$ (see, e.g.,\ 
\cite{BestvinaFeighn, DayPutmanComplex,HatcherSphereComplex, HatcherVogtmannFreeFactors, HatcherVogtmannCerf, KapovichLustig}).
The groups $\AutFB{n}{k}{l}$ form the simplex stabilizers for the action of $\Aut(F_{n,k,l})$ on the {\em complex
of partial bases}, which is one of these analogues.  In \cite{DayPutmanComplex}, the results in the current
paper are used to prove some topological results about this complex.  The simplex stabilizers of several other curve complex
analogues are closely related to the groups $\AutFB{n}{k}{l}$.
\end{remark}

\begin{remark}
The group $\AutFB{0}{k}{l}$ is known as the 
{\em pure symmetric automorphism group} of the free group $F_{0,k,l}$.  It was introduced by 
McCool \cite{McCoolSymmetric} and has been the subject of much investigation.  McCool was
also the first to study the general group $\AutFB{n}{k}{l}$, and in \cite{McCoolPresentation} he
proved that it was finitely presentable.
\end{remark}

\begin{remark}
In \cite{WahlFromModToAut}, Wahl introduced a slightly different definition of an automorphism group of a free
group with boundary.  These groups were further studied in papers of Hatcher-Wahl \cite{HatcherWahl} and Jensen-Wahl \cite{JensenWahl}.
The paper \cite{JensenWahl} also discussed the groups $\AutFB{n}{k}{l}$ we consider here and constructed presentations
for them; see Theorem \ref{theorem:jensenwahl} below.
Wahl's groups are abelian extensions of $\AutFB{n}{k}{l}$, and it is not hard to deduce from our theorems
analogous results for her groups.  In the context of the mapping class group, there are likewise two natural ways
to deal with boundary components.  One can require that mapping classes fix the boundary pointwise or merely require them
to fix each boundary component setwise.  In the former case (to which Wahl's definition is analogous), the Dehn twists about the boundary components are nontrivial,
while in the latter case (to which our definition is analogous) they are trivial.
\end{remark}

\begin{remark}
Another reasonable definition of the automorphism group of a free group with boundary would require
the automorphisms to fix the $y_i$ and $z_j$ on the nose (rather than merely up to conjugacy).  It
turns out that this leads to a much more poorly behaved theory -- see Theorem \ref{maintheorem:alternatekernel}
below.
\end{remark}

\paragraph{The Birman exact sequence.}
Let $N$ be the subgroup of $F_{n,k,l}$ normally generated by $Y$.  There is an evident isomorphism
$F_{n,k,l}/N \cong F_{n,0,l}$, and we will henceforth identify these groups.  We thus have a short
exact sequence
$$1 \longrightarrow N \longrightarrow F_{n,k,l} \longrightarrow F_{n,0,l} \longrightarrow 1.$$
The group $\AutFB{n}{k}{l}$ preserves $N$, so we get a map $\rho \co \AutFB{n}{k}{l} \rightarrow \Aut(F_{n,0,l})$
whose image is clearly contained in $\AutFB{n}{0}{l} < \Aut(F_{n,0,l})$.  The map
$\rho \co \AutFB{n}{k}{l} \rightarrow \AutFB{n}{0}{l}$ has a right inverse 
$\psi \co \AutFB{n}{0}{l} \rightarrow \AutFB{n}{k}{l}$ defined as follows.  Consider
$f \in \AutFB{n}{0}{l}$.  Define $\psi(f) \in \AutFB{n}{k}{l}$ to be the automorphism
which has the following behavior on a generator $s \in X \cup Y \cup Z$.
$$\psi(f)(s) = 
\begin{cases}
s & \text{if $s \in Y$,}\\
f(s) & \text{if $s \in X \cup Z$.}
\end{cases}$$
It is clear that this defines an element of $\AutFB{n}{k}{l}$ such that $\rho(\psi(f)) = f$.  We
conclude that $\rho$ is surjective.

Let $\BKer{n}{k}{l}$ be the kernel of $\rho$.  We thus have a split short exact sequence
\begin{equation}
\label{eqn:birmanaut}
1 \longrightarrow \BKer{n}{k}{l} \longrightarrow \AutFB{n}{k}{l} \stackrel{\rho}{\longrightarrow} \AutFB{n}{0}{l} \longrightarrow 1.
\end{equation}
The exact sequence \eqref{eqn:birmanaut} will be our analogue of the Birman exact sequence \eqref{eqn:birmanmod}.  

\paragraph{Generators for the kernel.}
To understand \eqref{eqn:birmanaut}, we must study the kernel group $\BKer{n}{k}{l}$.  We
begin by giving generators for it.  For distinct $v,w \in X \cup Y \cup Z$ and $\epsilon = \pm 1$, let
$\Mul{v^{\epsilon}}{w}$ and $\Con{v}{w}$ be the elements of $\Aut(F_{n,k,l})$ that are defined by
\[\Mul{v^{\epsilon}}{w}(s) =
\begin{cases}
wv       & \text{if $s=v$ and $\epsilon=1$}\\
vw^{-1}  & \text{if $s=v$ and $\epsilon=-1$}\\
s        & \text{otherwise}
\end{cases}
\quad \text{and} \quad
\Con{v}{w}(s) =
\begin{cases}
w v w^{-1} & \text{if $s=v$}\\
s          & \text{otherwise}
\end{cases}\]
for $s \in X \cup Y \cup Z$.
Our first result is as follows.
\begin{maintheorem}
\label{maintheorem:generators}
The group $\BKer{n}{k}{l}$ is generated by the finite set
\begin{align*}
S_K =
 & \Set{$\Mul{x^{\epsilon}}{y}$}{$x \in X$, $y \in Y$, $\epsilon = \pm 1$}\cup \Set{$\Con{z}{y}$}{$z \in Z$, $y \in Y$}\\
 &\quad \cup \Set{$\Con{y}{v}$}{$y \in Y$, $v \in X \cup Y \cup Z$, $y \neq v$}.
\end{align*}
\end{maintheorem}

\paragraph{The kernel is not finitely presentable.}
Theorem \ref{maintheorem:generators} might suggest to the reader that 
$\BKer{n}{k}{l}$ is a well-behaved group, but this hope is dashed by the following theorem.

\begin{maintheorem}
\label{maintheorem:notfinitepres}
If $k \geq 1$ and $n+l \geq 2$, then $\BKer{n}{k}{l}$ is not finitely presentable.
In fact, $H_2(\BKer{n}{k}{l};\Q)$ has infinite rank.
\end{maintheorem}

\begin{remark}
Theorem \ref{maintheorem:notfinitepres} should be contrasted to what happens in the mapping
class group of a surface, where the kernel of the Birman exact sequence is finitely presented (and, in fact, has a compact
Eilenberg-MacLane space, so all of its homology groups have finite rank).
\end{remark}

Our proof of Theorem \ref{maintheorem:notfinitepres} actually yields an explicit infinite
set of linearly independent classes in $\HH_2(\BKer{n}{k}{l};\Q)$.  These classes
arise as {\em abelian cycles}.  If $G$ is a group and $x,y \in G$ are commuting elements,
then there is a homomorphism $i \co \Z^2 \rightarrow G$ taking the generators of $\Z^2$
to $x$ and $y$.  The $2$-torus is an Eilenberg-MacLane space for $\Z^2$, so $\HH_2(\Z^2;\Q) \cong \Q$.
Letting $c \in \HH_2(\Z^2;\Q)$ be the standard generator, the element $i_{\ast}(c) \in \HH_2(G;\Q)$
is known as the abelian cycle determined by $x$ and $y$.  Returning to $\BKer{n}{k}{l}$,
the assumptions of Theorem \ref{maintheorem:notfinitepres} imply that we can find
$y \in Y$ and $a,b \in X \cup Z$ such that $a \neq b$.  It is easily verified that the elements
$$\Con{y}{a}^m \Con{b}{y} \Con{y}{a}^{-m} \quad \quad \text{and} \quad \quad \Con{a}{y} \Con{b}{y}$$
of $\BKer{n}{k}{l}$ commute for all $m \geq 1$.  Letting $\mu_m \in \HH_2(\BKer{n}{k}{l};\Q)$ be the
associated abelian cycle, we will prove that the $\mu_m$ are linearly independent.

\begin{remark}
A priori, it is not even clear that the $\mu_m$ are nonzero cohomology classes.
\end{remark}

\paragraph{Comment about proof of Theorem \ref{maintheorem:notfinitepres}.}
The key step in proving that the $\mu_m$ are linearly independent is the construction
of elements $\zeta_m \in \HH^2(\BKer{n}{k}{l};\Q)$ which
are ``dual'' to the $\mu_m$.  The $\zeta_m$ are ``almost'' cup products of elements of $H^1$.
More precisely, there is a natural subgroup $\LKernkl < \BKer{n}{k}{l}$ and a homomorphism
$I \co \LKernkl \rightarrow A$, where $A$ is an infinite rank abelian group.  The homomorphism $I$
can be viewed as a generalization of the {\em Johnson homomorphism}, which 
is a well-known abelian quotient of the Torelli subgroup of the mapping class group of a surface (see
\S \ref{section:generalizedjohnson} for details).  We construct a sequence of
surjections $\alpha_m' \co A \rightarrow \Z$, and thus a sequence of surjections
$\alpha_m \co \LKernkl \rightarrow \Z$.  The cohomology class $[\alpha_m]$ is then
an element of $\HH^1(\LKernkl;\Q)$, and we can construct elements 
$$\eta_m = [\alpha_m] \cup [\alpha_0] \in \HH^2(\LKernkl;\Q).$$
The elements $\zeta_m \in \HH^2(\BKer{n}{k}{l};\Q)$ are constructed from $\eta_m$ by
a sort of averaging process similar to the classical transfer map, though the
transfer map cannot be used directly since $\LKernkl$ is an infinite-index subgroup
of $\BKer{n}{k}{l}$.

\paragraph{Relations in the kernel.}
In spite of Theorem \ref{maintheorem:notfinitepres}, 
it turns out that there are only five basic relations in $\BKer{n}{k}{l}$.
These appear as relations R1-R5 in Table~\ref{table:RKrelations}.  We have the following theorem.

\begin{maintheorem}[Informal]
\label{maintheorem:relations}
The group $\BKer{n}{k}{l}$ has a presentation $\GroupPres{S_K}{R}$, where $S_K$ is 
the generating set from Theorem \ref{maintheorem:generators} and $R$ consists of all relations 
``of the same type'' as relations R1-R5 in Table~\ref{table:RKrelations}.
\end{maintheorem}

\begin{table}[ht]
\begin{tabular}{p{0.95\textwidth}}
\toprule
\begin{center}
{\bf Basic relations for $\BKer{n}{k}{l}$}
\end{center}

Let $S_K$ be the generating set from Theorem~\ref{maintheorem:generators}.
The set $R_K$ is the set of all relations between elements of $F(S_K)$ of the following forms:
\begin{itemize}
\item[R1.\ ] Three classes of relations saying that generators commute.
\begin{itemize}
\item[(1)\ ] $[\Mul{x^\epsilon}{v},\Mul{w^\delta}{y}]=1$
for $x,w\in X$, $\epsilon,\delta=\pm1$ and $v,y\in Y$ with $x^\epsilon\neq w^\delta$,
\item[(2)\ ] $[\Mul{x^\epsilon}{y},\Con{v}{z}]=1$
for $x\in X$, $\epsilon=\pm1$, $y\in Y$, $v\in Y\cup Z$ and $z\in X\cup Y\cup Z$ such that $\Con{v}{z}\in S_K$, $x\neq z$ and $v\neq y$,
\item[(3)\ ] $[\Con{u}{v},\Con{w}{z}]=1$
for $u,w\in Y\cup Z$ and $v,z\in X\cup Y\cup Z$ such that $\Con{u}{v}, \Con{w}{z}\in S_K$, $u\neq w, z$ and $w\neq v$;
\end{itemize}
\item[R2.\ ] $\Con{v}{x}^{-\epsilon} \Mul{x^{\epsilon}}{z} \Con{v}{x}^{\epsilon} =  \Con{v}{z}\Mul{x^{\epsilon}}{z}$ for $\epsilon = \pm 1$,
$x \in X$, and $v,z \in Y$ such that $v \neq z$;
\item[R3.\ ] $\Con{v}{z}^{\epsilon} \Mul{x^{\delta}}{v} \Con{v}{z}^{-\epsilon} = \Mul{x^{\delta}}{z}^{-\epsilon} \Mul{x^{\delta}}{v} \Mul{x^{\delta}}{z}^{\epsilon}$
for $\epsilon, \delta = \pm 1$, $x \in X$, and $v,z \in Y$ such that $v \neq z$;
\item[R4.\ ] $\Con{v}{z}^{\epsilon} \Con{w}{v} \Con{v}{z}^{-\epsilon} = \Con{w}{z}^{-\epsilon} \Con{w}{v} \Con{w}{z}^{\epsilon}$ for $\epsilon = \pm 1$, $w,v\in Y\cup Z$ and $z\in X\cup Y\cup Z$ such that $\Con{v}{z}$, $\Con{w}{v}$, $\Con{w}{z}\in S_K$; and
\item[R5.\ ] $\Con{y}{x}^{-\epsilon} \Mul{x^{-\epsilon}}{y} \Con{y}{x}^{\epsilon} = \Mul{x^{\epsilon}}{y}^{-1}$ for $\epsilon = \pm 1$, $x \in X$, and $y \in Y$.
\end{itemize}\\
\bottomrule
\end{tabular}
\caption{Relations for $\BKer{n}{k}{l}$}
\label{table:RKrelations}
\end{table}

\noindent
Of course, this appears to contradict Theorem \ref{maintheorem:notfinitepres}.  However,
the relations in R1-R5 depend on a choice of basis, and in
Theorem \ref{maintheorem:relations} we require
the infinite collection
of relations of the forms R1-R5 with respect to all choices of basis.
See Theorem \ref{theorem:relations} in \S \ref{section:generatorsandrelations} for a precise statement.
Also see Corollary~\ref{corollary:finiteLpresentation}, where we deduce that $\BKer{n}{k}{l}$ has a finite $L$-presentation (a strong kind of recursive presentation).

\paragraph{The abelianization of the kernel.}
Our next main theorem gives the abelianization of $\BKer{n}{k}{l}$.  There are three families of abelian
quotients (see \S \ref{section:abelianization} for more details)
\begin{itemize}
\item Let $\psi$ be the restriction of the natural homomorphism
$\Aut(F_{n,k,l}) \rightarrow \Aut(\Z^{n+k+l}) \cong \GL_{n+k+l}(\Z)$ to $\BKer{n}{k}{l}$.
Then it turns out that the image of $\psi$ is a free abelian subgroup of $\GL_{n+k+l}(\Z)$ of rank $nk$.
\item If $n=0$, then $\BKer{n}{k}{l}$ acts trivially on $F_{n,k,l}^{\Ab} = \Z^{n+k+l}$ and thus by
definition lies in the subgroup
\[\IA_{n,k,l} = \Set{$f \in \Aut(F_{n,k,l})$}{$f$ acts trivially on $F_{n,k,l}^{\Ab}$}\]
of $\Aut(F_{n,k,l})$.  The group $\IA_{n,k,l}$ is often called the {\em Torelli subgroup} of $\Aut(F_{n,k,l})$.
There is a well-known homomorphism $J \co \IA_n \rightarrow \Hom(\Z^{n+k+l}, \bigwedge^2 \Z^{n+k+l})$ known as the
{\em Johnson homomorphism}, and $J(\BKer{n}{k}{l})$ has rank $2kl + k(k-1)$.
\item If $n>0$, then $\BKer{n}{k}{l}$ does not act trivially on $F_{n,k,l}^{\Ab}$ so the
Johnson homomorphism is not available.  Nonetheless, we will construct a sequence of modified versions
of the Johnson homomorphism in this case the direct sum of whose images is a free abelian group of rank $2kl+kn$.
\end{itemize}
The above abelian quotients are all independent of each other, and we prove that they
give the entire abelianization of $\BKer{n}{k}{l}$.

\begin{maintheorem}
\label{maintheorem:abelianization}
In the case $n=0$, we have $\HH_1(\BKer{0}{k}{l};\Z) \cong \Z^{2kl+k(k-1)}$.  If $n > 0$, then
$\HH_1(\BKer{n}{k}{l};\Z) \cong \Z^{2kn+2kl}$.
\end{maintheorem}

\paragraph{An alternate definition.}
Our last theorem concerns a possible alternate definition of the automorphism group
of a free group with boundary.  Define
\[\AutFB{n}{k}{l}' = \Set{$f \in \Aut(F_{n,k,l})$}{$f(v) = v$ for $v \in Y \cup Z$}.\]
The group $\AutFB{n}{k}{l}'$ was first studied by McCool \cite{McCoolPresentation}, who proved that
it was finitely presentable.  Just as for $\AutFB{n}{k}{l}$, there is a split surjection
$\rho' \co \AutFB{n}{k}{l}' \rightarrow \AutFB{n}{0}{l}$.  Define $\BKer{n}{k}{l}' = \Ker(\rho')$,
so we have a short exact sequence
$$1 \longrightarrow \BKer{n}{k}{l}' \longrightarrow \AutFB{n}{k}{l}' \stackrel{\rho'}{\longrightarrow} \AutFB{n}{0}{l}' \longrightarrow 1.$$
One might expect that $\BKer{n}{k}{l}'$ is similar to $\BKer{n}{k}{l}$; however, the following theorem
shows that they are quite different.

\begin{maintheorem}
\label{maintheorem:alternatekernel}
Fix $n,k \geq 1$ and $l \geq 0$ such that $n+l \geq 2$.  Then the group 
$\BKer{n}{k}{l}'$ is not finitely generated.  In fact, $\HH_1(\BKer{n}{k}{l}';\Q)$ has infinite rank.
\end{maintheorem}

The proof of this falls out of our proof of Theorem \ref{maintheorem:notfinitepres} above.  Recall
from the proof sketch that one of the key steps was the construction of a subgroup $\LKernkl < \BKer{n}{k}{l}$
and a ``Johnson homomorphism'' $I \co \LKernkl \rightarrow A$, where $A$ is an abelian group of infinite rank.
It turns out that $\BKer{n}{k}{l}' \subset \LKernkl$ and the image $I(\BKer{n}{k}{l}')$ has infinite rank.

\paragraph{Prior results.}
Some special cases of our theorems appear in the literature on the pure symmetric automorphism
group.  Define $K_m = \BKer{0}{1}{m-1}$.  In \cite{CollinsGilbert}, Collins and Gilbert proved that
$K_3$ is not finitely presentable.  In \cite{Pettet}, Pettet proved that $K_m$ is finitely
generated for all $m$ and that $K_m$ is not finitely presentable for $m \geq 3$.  She also calculated
the abelianization of $K_m$.  Our proof of the finite generation of $\BKer{n}{k}{l}$ is
a generalization of the proof in \cite{Pettet}; however, our proofs of our other results are quite
different from those in \cite{CollinsGilbert} and \cite{Pettet}.

\paragraph{Acknowledgments.}
We wish to thank Nathalie Wahl for helpful correspondence and Tom Church for sending us some corrections.  We
also wish to thank the referee for many useful suggestions.

\paragraph{Outline and conventions.}
In \S \ref{section:generalizedjohnson}, we give a general construction of a ``Johnson crossed
homomorphism''.
This construction will be used in both \S \ref{section:abelianization} and \S \ref{section:notfinitepres}.
In \S \ref{section:abelianization}, we calculate the abelianization of $\BKer{n}{k}{l}$, proving
Theorem \ref{maintheorem:abelianization}.  
In \S \ref{section:notfinitepres}, we
prove Theorem \ref{maintheorem:notfinitepres}, which asserts that $\BKer{n}{k}{l}$ is not
finitely presentable if $k\geq 1$ and $n+l\geq 2$.  
The proof of Theorem \ref{maintheorem:alternatekernel}
is also contained in \S \ref{section:notfinitepres}.
Finally, in \S \ref{section:generatorsandrelations}, we determine generators and relations for $\BKer{n}{k}{l}$,
proving Theorems \ref{maintheorem:generators} and \ref{maintheorem:relations}. 
We put \S \ref{section:generatorsandrelations} at the end because of its heavy use of combinatorial group theory, however, it does not depend on the other sections and can also be read first.

Automorphisms act on the left and compose right-to-left like functions.
If $G$ is a group and $g,h \in H$, then we define
$g^h=hgh^{-1}$ and $[g,h]=g h g^{-1}h^{-1}$.
We note again our conventions that $\Con{x}{y}(x)=x^y$ and $\Mul{x}{y}(x)=yx$.

\section{Johnson crossed homomorphisms}
\label{section:generalizedjohnson}

To study $\HH_1(\BKer{n}{k}{l};\Z)$ and $\HH_2(\BKer{n}{k}{l};\Q)$, we will need
several variants on the well-known Johnson homomorphisms.  These homomorphisms
were originally introduced in the context of the Torelli subgroup of the mapping
class group of a surface by Johnson \cite{JohnsonHomo,JohnsonSurvey} and have
since been generalized to a variety of contexts (see, e.g., \cite{BroaddusFarbPutman, 
ChurchFarb, CohenPakianathan, FarbIA, Kawazumi}).  In this section, we give a
general framework for studying them.

Recall first that if a group $\Gamma$ acts on an abelian group $M$, then a {\em crossed homomorphism} 
from $\Gamma$ to $M$ with respect to this action
is a function $\phi \co \Gamma \rightarrow M$ satisfying the cocycle identity
$$\phi(g_1 g_2) = \phi(g_1) + g_1(\phi(g_2)) \quad \quad (g_1, g_2 \in \Gamma).$$
Crossed homomorphisms are also sometimes referred to as \emph{twisted 1-cocycles} or as \emph{derivations}.
The set of all crossed homomorphisms from $\Gamma$ to $M$ forms an abelian group which
we will denote $\CHom(\Gamma,M)$.  If $\Gamma$ acts trivially on $M$, then $\CHom(\Gamma,M) = \Hom(\Gamma,M)$. 

Now let
$$1 \longrightarrow A \longrightarrow B \longrightarrow C \longrightarrow 1$$
be a short exact sequence of groups and let $G$ be a group which acts on $B$.  Assume
the following three conditions hold.
\begin{enumerate}
\item The group $A$ is abelian.
\item For $a \in A$ and $g \in G$, we have $g(a) \in A$.
\item The induced action of $G$ on $C = B/A$ is trivial.
\end{enumerate}
To simplify our notation, we will often use additive notation when discussing $A$, though we
will try never to mix additive and multiplicative notation.  

For $g \in G$, define a function $\mathcal{J}_g \co B \rightarrow A$ as follows.  Condition 3 implies that 
$g(b) \cdot b^{-1} \in A$ for $b \in B$.  We can thus define $\mathcal{J}_g(b) = g(b) \cdot b^{-1}$.  
The group $B$ acts on $A$ by conjugation.  We claim that $\mathcal{J}_g$ is a crossed homomorphism.
Indeed, for $b_1,b_2 \in B$ we have
\begin{align*}
\mathcal{J}_g(b_1 \cdot b_2) &= g(b_1 \cdot b_2) \cdot (b_1 \cdot b_2)^{-1} \\
&= (g(b_1) \cdot b_1^{-1}) \cdot b_1 \cdot (g(b_2) \cdot b_2^{-1}) \cdot b_1^{-1} \\
&= \mathcal{J}_g(b_1) + b_1 (\mathcal{J}_g(b_2)).
\end{align*}
We can thus define a function
$$\mathcal{J} \co G \longrightarrow \CHom(B,A)$$
by $\mathcal{J}(g) = \mathcal{J}_g$.  The group $G$ acts on $\CHom(B,A)$ via the action
of $G$ on $A$.  We claim that $\mathcal{J}$ is a crossed homomorphism.  Indeed, for
$g_1,g_2 \in G$ and $b \in B$ we have
\begin{align*}
\mathcal{J}(g_1 \cdot g_2)(b) &= g_1(g_2(b)) \cdot b^{-1}\\
&= g_1(g_2(b) \cdot b^{-1}) \cdot g_1(b) \cdot b^{-1}\\
&= g_1(\mathcal{J}(g_2)(b)) + \mathcal{J}(g_1).
\end{align*}
We will call the function $\mathcal{J}$ a {\em Johnson crossed homomorphism}.

\section{The abelianization of $\BKer{n}{k}{l}$}
\label{section:abelianization}

In this section, we calculate the abelianization of $\BKer{n}{k}{l}$.
The actual calculation is contained in \S \ref{section:calculateabel}.
This is proceeded by \S \ref{section:homologyaction}--\ref{section:johnson}, which
are devoted to constructing the necessary abelian quotients
of $\BKer{n}{k}{l}$.  The homomorphisms constructed in \S \ref{section:homologyaction}
come from the action of $\BKer{n}{k}{l}$ on $F_{n,k,l}^{\Ab}$, and the
homomorphisms in \S \ref{section:johnson} come from certain Johnson
crossed homomorphisms.

To simplify our notation, we define
\[V = F_{n,k,l} = \Span{x_1,\ldots,x_n,y_1,\ldots,y_k,z_1,\ldots,z_l} \quad \text{and} \quad \overline{V} = V^{\Ab}\]
and
\[\Vc = \Span{y_1,\ldots,y_k,z_1,\ldots,z_l} \quad \text{and} \quad \BVc = \Vc^{\Ab}\]
and
\[\Vcy = \Span{y_1,\ldots,y_k} \quad \text{and} \quad \BVcy = \Vcy^{\Ab}.\]
Also, for $v \in V$ we will denote by $\overline{v} \in \overline{V}$ the image in the abelianization.
The notation $\BVc$ is intended to indicate that this is the subgroup of $V$ generated
by basis elements whose conjugacy classes are fixed by $\AutFB{n}{k}{l}$.

\subsection{The action on homology}
\label{section:homologyaction}

The first source of abelian quotients of $\BKer{n}{k}{l}$ is the action of 
$\AutFB{n}{k}{l}$ on $\overline{V}$.  This action restricts to the identity action on $\BVc$.  
If $M$ and $N$ are free $\Z$-modules and $N$ is a direct summand of $M$, then
denote by $\Aut(M,N)$ the set of automorphisms of $M$ that restrict to the identity on $N$.
With respect to an appropriate choice of basis, elements of $\Aut(M,N)$ are represented by 
matrices with an identity block in the upper left
hand corner and a block of zeros in the lower left hand corner.  Let 
$\pi \co \overline{V} \rightarrow \overline{V}/\BVc$ be the projection.
We have a split short exact sequence
\[\begin{CD}
1 @>>> \Hom(\overline{V}/\BVc, \BVcy) @>>> \Aut(\overline{V},\BVc) @>>> \Aut(\overline{V}/\BVcy,\BVc / \BVcy) @>>> 1,
\end{CD}\]
where $h \in \Hom(\overline{V}/\BVc, \BVcy)$ corresponds to the automorphism of 
$\overline{V}$ that takes $\overline{v} \in \overline{V}$
to $\overline{v} + h(\pi(\overline{v}))$.  This fits into a commutative diagram of the form
\[\begin{CD}
1 @>>> \BKer{n}{k}{l}                                            @>>> \AutFB{n}{k}{l}                   @>>> \AutFB{n}{0}{l}       @>>> 1 \\
@.     @VVV                                                           @VVV                                   @VVV                               @.\\
1 @>>> \Hom(\overline{V}/\BVc, \BVcy) @>>> \Aut(\overline{V},\BVc) @>>> \Aut(\overline{V}/\BVcy,\BVc / \BVcy) @>>> 1.
\end{CD}\]
In summary, we have a homomorphism $A \co \BKer{n}{k}{l} \rightarrow \Hom(\overline{V}/\BVc, \BVcy)$.

\begin{remark}
It is easy to write down a formula for $A$.  Namely, if $f \in \BKer{n}{k}{l}$ and $w \in \overline{V}/\BVc$, then
$A(f)(w) = f_{\ast}(\overline{v}) - \overline{v} \in \BVcy$, where $\overline{v} \in \overline{V}$ is any lift of $w$ and
$f_{\ast} : \overline{V} \rightarrow \overline{V}$ is the induced action.
\end{remark}

\subsection{The Johnson homomorphisms on $\BKer{n}{k}{l}$}
\label{section:johnson}

The second source of abelian quotients of $\BKer{n}{k}{l}$ are homomorphisms constructed
from Johnson crossed homomorphisms in the sense of \S \ref{section:generalizedjohnson}.
Let $\Vc' < V$ be the subgroup generated by $\Vc$ and $[V,V]$.  We then
have a short exact sequence
\begin{equation}
\label{eqn:vcprime1}
1 \longrightarrow [V,V] \longrightarrow \Vc' \longrightarrow \BVc \longrightarrow 1.
\end{equation}
It is well-known that $[V,V]/[V,[V,V]] \cong \bigwedge^2 \overline{V}$ (see, e.g., \cite{JohnsonHomo}).
Taking the quotient of the groups in the short exact sequence \eqref{eqn:vcprime1} by $[V,[V,V]]$,
we thus get a short exact sequence
\begin{equation}
\label{eqn:vcprime2}
1 \longrightarrow \WedgeTwo \overline{V} \longrightarrow \Vc'/[V,[V,V]] \longrightarrow \BVc \longrightarrow 1.
\end{equation}
It is clear that the group $\BKer{n}{k}{l}$ acts on $\Vc'$ and thus on $\Vc'/[V,[V,V]]$.  Also,
the action of $\BKer{n}{k}{l}$ on $\BVc$ is trivial.  By the method of \S \ref{section:generalizedjohnson},
we thus get a Johnson crossed homomorphism 
$J \co \BKer{n}{k}{l} \rightarrow \CHom(\BVc, \bigwedge^2 \overline{V})$.
The action of $\BVc$ on $\bigwedge^2 \overline{V}$ is trivial, so 
$\CHom(\BVc, \bigwedge^2 \overline{V}) = \Hom(\BVc, \bigwedge^2 \overline{V})$.  Summing up, we have
constructed a Johnson crossed homomorphism
$$J \co \BKer{n}{k}{l} \longrightarrow \Hom(\BVc, \WedgeTwo \overline{V}).$$

If $n=0$, then $\overline{V} = \overline{V}_C$ and thus $\BKer{n}{k}{l}$ acts trivially on $\bigwedge^2 \overline{V}$.  This
implies that in this case $J$ is actually a homomorphism.  If $n \geq 1$, however, then it is only a crossed
homomorphism.  This can be fixed as follows.  Consider $c \in Y \cup Z$.  Let 
$J_c \co \BKer{n}{k}{l} \rightarrow \bigwedge^2 \overline{V}$ be the composition
of $J$ with the restriction map from $\Hom(\BVc, \bigwedge^2 \overline{V})$
to
$$\Hom(\Span{\overline{c}}, \WedgeTwo \overline{V}) \cong \Hom(\Z,\WedgeTwo \overline{V}) \cong \WedgeTwo \overline{V}.$$
For $f \in \BKer{n}{k}{l}$, there exists some $w \in V$ such that $f(c) = w c w^{-1}$.  By
definition, $J_c(f)$ is the image of $f(c) \cdot c^{-1} = [w,c]$ in $\bigwedge^2 \overline{V}$.  Letting
$\overline{c}$ and $\overline{w}$ be the images of $c$ and $w$ in $\overline{V}$, respectively,
we get that $J_c(f) = \overline{w} \wedge \overline{c}$.  In other words, the image
of $J_c$ lies in $\overline{V} \wedge \overline{c} \subset \bigwedge^2 \overline{V}$, which 
is isomorphic to $\overline{V} / \Span{\overline{c}}$ 
via the isomorphism that takes $x \in \overline{V} / \Span{\overline{c}}$
to $\tilde{x} \wedge \overline{c}$, where $\tilde{x} \in \overline{V}$ is any lift of $x$.

So we have described a crossed homomorphism 
\[J_c\co \BKer{n}{k}{l}\longrightarrow \overline{V}/\Span{\overline{c}}.\]
If we further suppose that $c\in Z$, then the fact that $f$ is in $\BKer{n}{k}{l}$ implies that $\overline{w}$ is in $\Span{\BVcy,\overline{c}}$
(otherwise $wcw^{-1}$ will not become $c$ when the elements of $Y$ are deleted from it).
So in the case that $c\in Z$, we have
\[J_c\co \BKer{n}{k}{l}\longrightarrow \Span{\BVcy,\overline{c}}/\Span{\overline{c}}.\]
Since $\Span{\BVcy,\overline{c}} \subset \BVc$, the group $\BKer{n}{k}{l}$ acts trivially on $\Span{\BVcy,\overline{c}}/\Span{\overline{c}}$.
This implies that $J_c$ is a homomorphism when $c\in Z$.  Tracing through the above identifications, we
may calculate $J_c(f)$ for $f \in \BKer{n}{k}{l}$ as follows.  Write $f(c) = w c w^{-1}$ for some
$w \in V$.  Then $\overline{w} \in \Span{\BVcy,\overline{c}}$ and $f(c)$ is the image of $\overline{w}$ in $\Span{\BVcy,\overline{c}}/\Span{\overline{c}}$.

On the other hand, if $c\in Y$, then we are forced to take a further quotient to get an honest homomorphism.
By definition $\BKer{n}{k}{l}$ acts trivially on $\overline{V} / \BVcy$.
Letting
$$J'_c \co \BKer{n}{k}{l} \rightarrow \overline{V} / \Span{\BVcy, \overline{c}}$$
be the composition of $J_c$ (with its target redefined via the above isomorphism to
be $\overline{V} / \Span{\overline{c}}$) with the projection 
$\overline{V} / \Span{\overline{c}} \rightarrow \overline{V} / \Span{\BVcy, \overline{c}}$,
it follows that $J'_c$ is a homomorphism.  Tracing through the above identifications,
we may calculate $J'_c(f)$ for $f \in \BKer{n}{k}{l}$ as follows.  Write
$f(c) = w c w^{-1}$ for some $w \in V$.  Then $J'_c(f)$ is the image of $w$ in
$\overline{V} / \Span{Y, \overline{c}}$.

\subsection{Calculating the abelianization}
\label{section:calculateabel}

We now have all the pieces necessary to prove the following two theorems.  Together they
imply Theorem \ref{maintheorem:abelianization}.

\begin{theorem}[{Abelianization of $\BKer{n}{k}{l}$, $n = 0$}]
Consider $k \geq 1$ and $l \geq 0$.  The image of the homomorphism
\[J \co \BKer{0}{k}{l} \longrightarrow \Hom(\overline{V}, \WedgeTwo \overline{V})\]
is the abelianization of $\BKer{0}{k}{l}$.  This image is isomorphic to $\Z^r$ with $r=2kl+k(k-1)$.
\end{theorem}
\begin{proof}
Define
\begin{align*}
&C = \Set{$\Con{y_i}{y_j}$}{$1 \leq i,j \leq k$, $i \neq j$} \quad \text{and} \quad C' = \Set{$\Con{z_i}{y_j}$}{$1 \leq i \leq l$, $1 \leq j \leq k$} \\
&\quad \quad \quad \quad \quad \text{and} \quad C'' = \Set{$\Con{y_i}{z_j}$}{$1 \leq i \leq k$, $1 \leq j \leq l$}.
\end{align*}
By Theorem \ref{maintheorem:generators} (proved in \S \ref{section:generatorsandrelations} below), the group $\BKer{0}{k}{l}$ is generated by $C \cup C' \cup C''$.  Since
$C \cup C' \cup C''$ has $k(k-1) + 2kl$ elements, it is enough to prove that the images under
$J$ of the elements of $C \cup C' \cup C''$ are linearly independent.  This follows from
the following three calculations.  Checking these calculations is easy and left to the reader.
\begin{itemize}
\item Consider $\Con{y_i}{y_j} \in C$.  Then $J(\Con{y_i}{y_j}) \in \Hom(\overline{V},\bigwedge^2 \overline{V})$
is the map with the following behavior for $s \in \{y_1,\ldots,y_k,z_1,\ldots,z_l\}$.
\[J(\Con{y_i}{y_j})(\overline{s}) =
\begin{cases}
\overline{y}_i \wedge \overline{y}_j & \text{if $s = \overline{y}_i$,}\\
0 & \text{otherwise.}\\
\end{cases}\]
\item Consider $\Con{z_i}{y_j} \in C'$.  Then $J(\Con{z_i}{y_j}) \in \Hom(\overline{V},\bigwedge^2 \overline{V})$
is the map with the following behavior for $s \in \{y_1,\ldots,y_k,z_1,\ldots,z_l\}$.
\[J(\Con{z_i}{y_j})(\overline{s}) =
\begin{cases}
\overline{z}_i \wedge \overline{y}_j & \text{if $s = \overline{z}_i$,}\\
0 & \text{otherwise.}\\
\end{cases}\]
\item Consider $\Con{y_i}{z_j} \in C$.  Then $J(\Con{y_i}{z_j}) \in \Hom(\overline{V},\bigwedge^2 \overline{V})$
is the map with the following behavior for $s \in \{y_1,\ldots,y_k,z_1,\ldots,z_l\}$.
\[J(\Con{y_i}{z_j})(\overline{s}) =
\begin{cases}
\overline{y}_i \wedge \overline{z}_j & \text{if $s = \overline{y}_i$,}\\
0 & \text{otherwise.}\\
\end{cases}\]
\end{itemize}
\end{proof}

\begin{theorem}[{Abelianization of $\BKer{n}{k}{l}$, $n \geq 1$}]
\label{theorem:ngeq1abelianization}
Consider $n \geq 1$, $k \geq 1$, and $l \geq 0$.  The image of the homomorphism
\[A \oplus (\bigoplus_{c \in Y} J_c')\oplus (\bigoplus_{c \in Z} J_c) \co \BKer{n}{k}{l} \longrightarrow \Hom(\overline{V}/\BVc, \BVcy) \oplus 
(\bigoplus_{c \in Y} \overline{V} / \Span{\BVcy, \overline{c}})\oplus 
(\bigoplus_{c \in Z} \Span{\BVcy, \overline{c}}) / \Span{\overline{c}})\]
is the abelianization of $\BKer{n}{k}{l}$.  This image is isomorphic to $\Z^r$ for $r = 2kn + kl$.
\end{theorem}
\begin{proof}
Define
\[M = \Set{$\Mul{x_i}{y_j}$}{$1 \leq i \leq n$, $1 \leq j \leq k$} \quad \text{and} \quad M' = \Set{$\Mul{x_i^{-1}}{y_j}$}{$1 \leq i \leq n$, $1 \leq j \leq k$}\]
and
\begin{align*}
&C = \Set{$\Con{y_i}{y_j}$}{$1 \leq i,j \leq k$, $i \neq j$} \quad \text{and} \quad C' = \Set{$\Con{z_i}{y_j}$}{$1 \leq i \leq l$, $1 \leq j \leq k$} \\
&\quad \text{and} \quad C'' = \Set{$\Con{y_i}{z_j}$}{$1 \leq i \leq k$, $1 \leq j \leq l$} \quad \text{and} \quad C''' = \Set{$\Con{y_i}{x_j}$}{$1 \leq i \leq k$, $1 \leq j \leq n$}.
\end{align*}
By Theorem \ref{maintheorem:generators}, the group $\BKer{n}{k}{l}$ is generated by $M \cup M' \cup C \cup C' \cup C'' \cup C'''$.

We have relations
\[\Mul{x_i}{y_j} \Mul{x_i^{-1}}{y_j} = [\Mul{x_i}{y_j}, \Con{y_j}{x_i}] \quad \quad \text{for $1 \leq i \leq n$ and $1 \leq j \leq k$}\]
and
\[\Con{y_i}{y_j} = [\Con{y_i}{x_1}^{-1},\Mul{x_1}{y_j}] \quad \quad \text{for $1 \leq i,j \leq k$ such that $i \neq j$}.\]
When $\BKer{n}{k}{l}$ is abelianized, therefore, the generators in $M'$ and $C$ become redundant.  Since $M$ has $kn$ elements, $C'''$ has $kn$ elements, and
$C'$ and $C''$ have $kl$ elements, it follows that the rank of the abelianization of $\BKer{n}{k}{l}$ is at most $2kn+2kl$.  To prove the theorem, it is thus
enough to show that the images under our map of the elements of $M \cup C' \cup C'' \cup C'''$ are linearly independent.  This follows from the following
four calculations.  For $c \in Y \cup Z$, let 
$q_c \co \Span{\BVcy, \overline{c}} \rightarrow \Span{\BVcy, \overline{c}}/ \Span{\overline{c}}$ and
$q'_c \co \overline{V} \rightarrow \overline{V}/ \Span{\BVcy, \overline{c}}$ be the natural maps.
\begin{itemize}
\item Clearly $J_c'(M) = 0$ for $c \in Y$ and $J_c(M)=0$ for $c\in Z$.  
Also, the elements of $M$ project under $A$ to a basis for the abelian group $\Hom(\overline{V}/\BVc, \BVcy)$.
\item Consider $\Con{z_i}{y_j} \in C'$.  Then $A(\Con{z_i}{y_j}) = 0$,  $J_c'(\Con{z_i}{y_j})=0$
for $c \in Y$ and $J_c(\Con{z_i}{y_j})=0$ for $c\in Z \setminus \{z_i\}$.  
Also, $J_{z_i}(\Con{z_i}{y_j}) = q_{z_i}(y_j)\neq 0$.
\item Consider $\Con{y_i}{z_j} \in C''$.  Then $A(\Con{y_i}{z_j}) = 0$, $J_c(\Con{y_i}{z_j})=0$ for $c\in Z$ and $J_c'(\Con{y_i}{z_j})=0$
for $c \in Y \setminus \{y_i\}$.  Also, $J_{y_i}'(\Con{y_i}{z_j}) = q'_{y_i}(z_j)\neq 0$.
\item Consider $\Con{y_i}{x_j} \in C'''$.  Then $A(\Con{y_i}{x_j}) = 0$, $J_c(\Con{y_i}{x_j})=0$ for $c\in Z$ and $J_c'(\Con{y_i}{x_j})=0$
for $c \in Y \setminus \{y_i\}$.  Also, $J_{y_i}'(\Con{y_i}{x_j}) = q'_{y_i}(x_j)\neq 0$.
\end{itemize}
\end{proof}

\section{The groups $\BKer{n}{k}{l}$ are not finitely presentable}
\label{section:notfinitepres}

The goal of this section is to prove Theorem \ref{maintheorem:notfinitepres}.  Along the way (in
\S \ref{section:lker}),
we will also prove Theorem \ref{maintheorem:alternatekernel}.  Let us
recall the setup of Theorem \ref{maintheorem:notfinitepres}.  We assume $n,k,l \in \Z$ are chosen
such that $k \geq 1$ and $n+l \geq 2$.  This implies that there exists some $y \in Y$ and
$a,b \in X \cup Z$ such that $a \neq b$.  Consider $m \geq 1$.  The elements
$$\Con{y}{a}^m \Con{b}{y} \Con{y}{a}^{-m} \quad \quad \text{and} \quad \quad \Con{a}{y} \Con{b}{y}$$
of $\BKer{n}{k}{l}$ commute and thus determine an abelian cycle $\mu_m \in \HH_2(\BKer{n}{k}{l};\Q)$.
Theorem \ref{maintheorem:notfinitepres} asserts that $\HH_2(\BKer{n}{k}{l};\Q)$ has
infinite rank, which will follow from the fact that the $\mu_m$ are linearly independent.
The skeleton of the proof of this is in \S \ref{section:skeletonproof}, which concludes
with an outline of the remainder of this section.
Because we will not change our hypotheses on $n$, $k$ and $l$ in this section, we will abbreviate $\BKer{n}{k}{l}$ as $\BKernkl$.

\subsection{Skeleton of proof}
\label{section:skeletonproof}

Recall that there is a bilinear pairing (the \emph{cap product} or \emph{evaluation pairing})
$$\omega \co \HH^2(\BKernkl;\Q) \times \HH_2(\BKernkl;\Q) \longrightarrow \Q.$$
In \S \ref{section:lker} -- \ref{section:evaluatezeta}, we will prove the following
lemma.

\begin{lemma}
\label{lemma:zeta}
For $r \geq 1$, there exist elements $[\zeta_r] \in \HH^2(\BKernkl;\Q)$ such that
$$\omega([\zeta_r],\mu_m) =
\begin{cases}
2 & \text{if $m = r$,}\\
0 & \text{if $m \neq r$,}
\end{cases}$$
for all $m \geq 1$.
\end{lemma}

\begin{remark}
We denote our cohomology class $[\zeta_r]$ because it will be defined by an explicit
cocycle $\zeta_r$ on $\BKernkl$.
\end{remark}

Lemma \ref{lemma:zeta} immediately implies Theorem \ref{maintheorem:notfinitepres}.  Indeed, if
$$\sum_{i=1}^{\infty} c_i \mu_i = 0 \quad \quad \text{($c_m \in \Q$, only finitely many $c_m$ nonzero)},$$
then we obtain that
$$0 = \omega(\sum_{i=1}^{\infty} [\zeta_r], c_i \mu_i) = \sum_{i=1}^{\infty} c_i \omega([\zeta_r], \mu_i) = 2 c_r$$
for all $r \geq 1$, as desired.

The outline of the proof of Lemma \ref{lemma:zeta} is as follows.  In \S \ref{section:lker}, we will construct
a certain subgroup $\LKernkl < \BKernkl$ and a family of homomorphisms from $\LKernkl$
to an infinite rank free abelian group $A$.  These homomorphisms are constructed from a Johnson
crossed homomorphism.  This section also contains a proof of Theorem \ref{maintheorem:alternatekernel}.
In \S \ref{section:zeta}, we will compose one of these homomorphisms
with a family of maps from $A$ to $\Z$ to obtain a sequence of surjective homomorphisms
$$\alpha_i \co \LKernkl \longrightarrow \Z \quad \quad \text{($i \geq 0$)}.$$
The maps $\alpha_i$ determine elements $[\alpha_i] \in \HH^1(\LKernkl;\Q)$, so we can use the cup product
to obtain cohomology classes $[\alpha_r] \cup [\alpha_0] \in \HH^2(\LKernkl;\Q)$ for $r \geq 1$.
We will use a sort of averaging process to construct
an element $[\zeta_r] \in \HH^2(\BKernkl;\Q)$ out of 
$[\alpha_r] \cup [\alpha_0] \in \HH^2(\LKernkl;\Q)$.
The calculation in Lemma \ref{lemma:zeta} is then contained in \S \ref{section:evaluatezeta}.  
Our calculations in \S \ref{section:zeta}--\ref{section:evaluatezeta} 
will (alas) have to be performed at the chain level, so between \S \ref{section:lker}
and \S \ref{section:zeta} we have \S \ref{section:groupchains}, which reviews some 
chain-level information about group cohomology.

\subsection{The group $\LKernkl$ and its Johnson homomorphism}
\label{section:lker}
The goal of this section is to define a subgroup $\LKernkl$ of $\BKernkl$ with infinitely many homomorphisms $\LKernkl\to\Z$; we use these to construct our desired cocycles on $\BKernkl$.
In fact, we define a crossed homomorphism from $\BKernkl$ to an infinite rank abelian group, and by restriction get homomorphisms on a subgroup.
In order to make these definitions, we must first investigate some subgroups of $V$ that arise naturally given the special role of $Y$.

Recall that we have fixed some $y \in Y$.  Also, recall that 
$$V = F_{n,k,l} \quad \quad \text{and} \quad \quad \overline{V} = V^{\Ab} \quad \quad \text{and} \quad \quad V_Y = \Span{Y} \subset V \quad \quad \text{and} \quad \quad \overline{V}_Y = V_Y^{\Ab}.$$
Define $N_Y$ to be the normal closure of $Y$ in $V$.  We have a short exact sequence
\begin{equation}
\label{eqn:nyvseq}
1 \longrightarrow N_Y \longrightarrow V \longrightarrow V/N_Y \longrightarrow 1.
\end{equation}
The group $\BKernkl$ acts on $V$.  This action preserves $N_Y$, and the induced action
on $V/N_Y$ is trivial by the definition of $\BKernkl$.  To construct a Johnson crossed homomorphism out of this
situation, we will need to quotient the groups in \eqref{eqn:nyvseq} by a subgroup $M \lhd V$ which
is preserved by $\BKernkl$ such that $M < N_Y$ and such that $N_Y /M$ is abelian.  Of
course, we could use $[N_Y,N_Y]$, but it turns out that a larger $M$ gives a Johnson crossed
homomorphism with better properties.

Let $Y' = Y \setminus \{y\}$ and $N_{Y'} < V$ be the normal closure of $Y'$.  Observe that
each of the subgroups
\begin{equation}
\label{eqn:mdefn}
N_{Y'} \quad \quad \text{and} \quad \quad [N_Y,N_Y] \quad \quad \text{and} \quad \quad [N_Y,[V,V]]
\end{equation}
is normal and preserved by $\BKernkl$.  Let $M < V$ be the subgroup generated by the subgroups in
\eqref{eqn:mdefn}.  It is clear that $M$ is a normal subgroup.  We have a short exact sequence
$$1 \longrightarrow N_Y/M \longrightarrow V/M \longrightarrow V/N_Y \longrightarrow 1.$$
Associated to this is a Johnson crossed homomorphism 
$$I \co \BKernkl \longrightarrow \CHom(V/M,N_Y/M).$$

Our next order of business is to find a subgroup $\LKernkl < \BKernkl$ such that
$\LKernkl$ acts trivially on $N_Y/M$.  The restriction of $I$ to $\LKernkl$ will then
be an actual homomorphism.  We need the following lemma about $N_Y / M$.

\begin{lemma}
\label{lemma:yvwelldefined}
If $\tilde{v} \in V$ and $\tilde{v}' \in V$ have the same image in $\overline{V} / \overline{V}_Y$,
then the images of $y^{\tilde{v}}$ and $y^{\tilde{v}'}$ in $N_Y/M$ are equal.
\end{lemma}
\begin{proof}
By assumption, we can find $w_1 \in [V,V]$ and $w_2 \in N_Y$ such that 
$\tilde{v} = \tilde{v}' w_1 w_2$.  It then follows that
\begin{equation}
\label{eqn:yveqn}
y^{\tilde{v}} = y^{\tilde{v}' w_1 w_2} = ((y^{w_2})^{w_1})^{\tilde{v}'}.
\end{equation}
Modulo $M$, we have $y^{w_2} = y$ and $y^{w_1} = y$.  It follows that modulo $M$, the
expression in \eqref{eqn:yveqn} equals $y^{\tilde{v}'}$, as desired.
\end{proof}

For $v \in \overline{V} / \overline{V}_Y$, pick some $\tilde{v} \in V$ that projects to $v$ and
define $y^v$ to equal the image of $y^{\tilde{v}}$ in $N_Y/M$.  
Lemma \ref{lemma:yvwelldefined} implies that $y^v$ is independent of the choice of 
$\tilde{v}$.  We then have the following.

\begin{lemma}
\label{lemma:yvbasis}
The infinite set $\{\text{$y^v$ $|$ $v \in \overline{V} / \overline{V}_Y$}\}$ is a basis
for the abelian group $N_Y/M$.
\end{lemma}
\begin{proof}
Let $A$ be the free abelian group with basis the formal symbols
$\{\text{$y^v$ $|$ $v \in \overline{V} / \overline{V}_Y$}\}$.
Using standard ``covering spaces of graphs'' arguments, it is easy to see that $N_Y$ is a free
group with free basis $S=\{\text{$y^{\tilde{v}}$ $|$ $\tilde{v} \in \Span{X,Z}$}\}$.  Let
$\rho \co N_Y \rightarrow A$ be the projection taking $y^{\tilde{v}} \in S$ to $y^v \in A$, where
$v \in \overline{V} / \overline{V}_Y$ is the image of $S$.  It is clear that $M < \Ker(\rho)$,
and Lemma \ref{lemma:yvwelldefined} implies that $\Ker(\rho) < M$.  The lemma follows.
\end{proof}

Since $\BKernkl$ preserves $M$, it acts on $N_Y/M$.  Recall that in \S \ref{section:johnson},
we constructed a homomorphism $J'_y \co \BKernkl \rightarrow \overline{V} / \overline{V}_Y$.
Tracing through the definitions, we see that
$$f(y^v) = y^{v + J'_y(f)} \quad \quad (\text{$v \in \overline{V}/\overline{V}_Y$, $f \in \BKernkl$})$$
We define $\LKernkl = \Ker(J'_y)$.
This lets us obtain the following result.

\begin{corollary}
\label{corollary:lkertrivial}
The group $\LKernkl < \BKernkl$ acts trivially on $N_Y/M$, and hence the restriction
of the Johnson crossed homomorphism $I \co \BKernkl \longrightarrow \CHom(V/M,N_Y/M)$
to $\LKernkl$ is a homomorphism.
\end{corollary}

The target $\CHom(V/M,N_Y/M)$ of $I$ is rather complicated, so our next order of business
is to break it into understandable pieces.  The key will be the following lemma, whose
proof is trivial and thus omitted.

\begin{lemma}
Let $G$ be an infinite cyclic group with generator $t$.  Also, let $A$ be an abelian
group upon which $G$ acts.  The map $i \co \CHom(G,A) \rightarrow A$ defined by
$i(\phi) = \phi(t)$ is then an isomorphism.
\end{lemma}

For $s \in X \cup Z$, its image $s' \in V/M$ generates an infinite cyclic subgroup.  We can
thus compose $I \co \LKernkl \longrightarrow \CHom(V/M,N_Y/M)$ with the restriction
map
$$\CHom(V/M,N_Y/M) \longrightarrow \CHom(\Span{s'},N_Y/M) \cong N_Y/M$$
to obtain a homomorphism $I_s \co \LKernkl \rightarrow N_Y/M$.  The value of $I_s$ is easy to calculate.
Namely, for $f \in \LKernkl$ we have $f(s) \cdot s^{-1} \in N_Y$, and $I_s(f)$ is the
image of $f(s) \cdot s^{-1}$ in $N_Y/M$.

At this point, we have developed enough machinery to prove Theorem \ref{maintheorem:alternatekernel}, 
whose we recall statement is as follows.  Recall that
\[\AutFB{n}{k}{l}' = \Set{$f \in \Aut(F_{n,k,l})$}{$f(v) = v$ for $v \in Y \cup Z$}.\]
There is an obvious surjective map $\rho \co \AutFB{n}{k}{l}' \rightarrow \AutFB{n}{0}{l}'$, and
$\BKer{n}{k}{l}' = \Ker(\rho)$.  Theorem \ref{maintheorem:alternatekernel} asserts that
$\HH_1(\BKer{n}{k}{l}';\Q)$ has infinite rank if $n,k \geq 1$ and $n+l \geq 2$.

\begin{proof}[{Proof of Theorem \ref{maintheorem:alternatekernel}}]
By assumption, we can find $s \in X$ and $t \in X \cup Z$ such that $s \neq t$.
Observe that $\BKer{n}{k}{l}' \subset \LKernkl$.  It follows that $I_s$ restricts
to a homomorphism $I_s \co \BKer{n}{k}{l}' \rightarrow N_Y/M$.  It is enough to 
show that $I_s(\BKer{n}{k}{l}')$ contains an infinite set of independent elements.
Now, for $m \in \Z$ define $h_m = \Con{y}{t}^m \Mul{s}{y} \Con{y}{t}^{-m}$.  It is easily verified
that $h_m \in \BKer{n}{k}{l}'$.  Also, $I_s(h_m)$ is the image of
$$h_m(s) s^{-1} = (t^m y t^{-m} s) s^{-1} = t^m y t^{-m}$$
in $N_Y/M$.  Letting $T$ be the image of $t$ in $\overline{V}/\overline{V}_Y$, we see
that $I_s(h_m) = y^{m T}$.  The set $\{\text{$y^{m T}$ $|$ $m \in \Z$}\}$
is an infinite set of independent elements of $N_Y/M$, and we are done.
\end{proof}

\subsection{Group cohomology at the chain level}
\label{section:groupchains}

This section contains some standard results about group cohomology, all of which
are contained in \cite{BrownCohomology} (although with different notation).  

\paragraph{The standard cochain complex.}
Let $G$ be a group.  Denote by
$\CC^n(G;\Q)$ the $\Q$-vector space consisting of all functions $\phi \co G^{n+1} \rightarrow \Q$ which
are $G$-equivariant in the sense that 
$\phi(g \cdot g_0, \ldots, g \cdot g_n) = \phi(g_0,\ldots,g_n)$ for all
$g,g_0,\ldots,g_n \in G$.  Elements of $\CC^n(G;\Q)$ are known as {\em $n$-cochains} on $G$.
They fit into a cochain complex
\begin{equation}
\label{eqn:gcochain}
\begin{CD}
\cdots @<<< \CC^{n+1}(G;\Q) @<<< \CC^n(G;\Q) @<<< \CC^{n-1}(G;\Q) @<<< \cdots\\
\end{CD}
\end{equation}
whose differential $\delta \co \CC^n(G;\Q) \rightarrow \CC^{n+1}(G;\Q)$ is given by
$$\delta(\phi)(g_0,\ldots,g_{n+1}) = \sum_{i=0}^{n+1} (-1)^i \phi(g_0,\ldots,\hat{g}_i,\ldots,g_{n+1}).$$
The cohomology groups of \eqref{eqn:gcochain} are $\HH^{\ast}(G;\Q)$.  If $\phi \in \CC^n(G;\Q)$
is a cocycle, then we will denote by $[\phi]$ the associated element of $\HH^n(G;\Q)$.

\begin{remark}\label{re:average}
The standard cochain complex is a complex of homomorphisms from the \emph{standard resolution} of
$\Z$ over $\Z G$ to $\Q$.
The standard resolution is a chain complex whose $n^{\text{th}}$ graded piece has a basis consisting of $(n+1)$-tuples of elements of $G$.
We note that there is a common alternate notation for the simplices in the standard resolution known as ``bar notation.''
The use of bar notation leads to a different description of the standard cochain complex and a different formula for the coboundary.
We are \emph{not} using bar notation; our reason for avoiding it is that our averaging construction is clearer without it.
\end{remark}

\paragraph{First cohomology.}
The $\Q$-vector space $\HH^1(G;\Q)$ is isomorphic to the space of homomorphisms
$G \rightarrow \Q$.  Given a homomorphism $f \co G \rightarrow \Q$, the associated
element of $\CC^1(G;\Q)$ is given by
$$\phi(g_0,g_1) = f(g_1) - f(g_0) \quad \quad (g_0,g_1 \in G).$$
We will denote the element of $\HH^1(G;\Q)$ corresponding to $f$ by $[f]$.

\paragraph{Cup products.}
There is a cup product map $\cup \co \HH^p(G;\Q) \otimes \HH^q(G;\Q) \rightarrow \HH^{p+q}(G;\Q)$.
If $\phi \in \CC^p(G;\Q)$ and $\phi' \in \CC^q(G;\Q)$ are cocycles, then $[\phi] \cup [\phi'] = [\phi'']$,
where $\phi'' \in \CC^{p+q}(G;\Q)$ is the cochain given by
$$\phi''(g_0,\ldots,g_{p+q}) = \phi(g_0,\ldots,g_p) \phi'(g_p,\ldots,g_{p+q}).$$

\paragraph{Evaluating on abelian cycles.}
For all $n \geq 0$, there is a bilinear pairing
$$\omega \co \HH^n(G;\Q) \times \HH_n(G;\Q) \rightarrow \Q.$$
We will need a formula for this in the following situation.  Let $f_1,f_2 \co G \rightarrow \Z$
be homomorphisms.  Also, let $g_1,g_2 \in G$ be commuting elements.  There is 
then a homomorphism $\Z^2 \rightarrow G$ taking
the generators of $\Z^2$ to the $g_i$.  Let $c \in \HH_2(G;\Q)$ be the image
of the standard generator of $\HH_2(\Z^2;\Q) \cong \Q$ under the induced map
$\HH_2(\Z^2;\Q) \rightarrow \HH_2(G;\Q)$.  We will call $c$ the {\em abelian cycle}
determined by $g_1$ and $g_2$.  Using the cup product structure on $\HH^{\ast}(\Z^2;\Z)$, we then
have
$$\omega([f_1] \cup [f_2],c) = f_1(g_1) f_2(g_2) - f_1(g_2) f_2(g_1).$$

\paragraph{Relating different groups.}
Assume now that $H$ is a subgroup of $G$.  We then have a diagram
\begin{equation}
\label{eqn:ggammacochain}
\begin{CD}
\cdots @<<< \CC^{n+1}(G;\Q) @<<< \CC^n(G;\Q) @<<< \CC^{n-1}(G;\Q) @<<< \cdots\\
@.          @VVV                 @VVV             @VVV                 @.\\
\cdots @<<< \CC^{n+1}(H;\Q) @<<< \CC^n(H;\Q) @<<< \CC^{n-1}(H;\Q) @<<< \cdots\\
\end{CD}
\end{equation}
of cochain complexes inducing the natural map $\HH^{\ast}(G;\Q) \rightarrow \HH^{\ast}(H;\Q)$.  Here
the map $\CC^n(G;\Q) \rightarrow \CC^n(H;\Q)$ is simply restriction.  

There is a cochain complex naturally lying between the cochain complexes
in \eqref{eqn:ggammacochain}.  Denote by $\CC^n(G,H;\Q)$ the $\Q$-vector space consisting of all functions 
$\phi \co G^{n+1} \rightarrow \Q$ that are $H$-equivariant in the sense that 
$\phi(g \cdot g_0, \ldots, g \cdot g_n) = \phi(g_0,\ldots,g_n)$ for all
$g \in H$ and $g_0,\ldots,g_n \in G$.  Observe that $\CC^n(G;\Q) \subset \CC^n(G,H;\Q)$.
The spaces $\CC^n(G,H;\Q)$ fit into a cochain complex
$$\begin{CD}
\cdots @<<< \CC^{n+1}(G,H;\Q) @<<< \CC^n(G,H;\Q) @<<< \CC^{n-1}(G,H;\Q) @<<< \cdots\\
\end{CD}$$
whose cohomology groups are easily seen to equal $\HH^{\ast}(H;\Q)$.  
Further, we have maps $\CC^n(G,H;\Q) \rightarrow \CC^n(H;\Q)$ by restriction. 
So we have a commutative diagram of the form
\begin{equation}
\label{eqn:ggammacochain2}
\begin{CD}
\cdots @<<< \CC^{n+1}(G;\Q) @<<< \CC^n(G;\Q) @<<< \CC^{n-1}(G;\Q) @<<< \cdots\\
@.          @VVV                 @VVV             @VVV                 @.\\
\cdots @<<< \CC^{n+1}(G,H;\Q) @<<< \CC^n(G,H;\Q) @<<< \CC^{n-1}(G,H;\Q) @<<< \cdots\\
@.          @VVV                 @VVV             @VVV                 @.\\
\cdots @<<< \CC^{n+1}(H;\Q) @<<< \CC^n(H;\Q) @<<< \CC^{n-1}(H;\Q) @<<< \cdots\\
\end{CD}
\end{equation}
The maps $\CC^n(G,H;\Q) \rightarrow \CC^n(H;\Q)$ induce a chain homotopy equivalence.  We
will need an explicit formula for an inverse chain homotopy equivalence.  Fix a set of right
coset representatives for $H$ in $G$.  Assume that the representative of the trivial
coset is $1$.  For $g \in G$, denote by $\hat{g}$ the coset representative of the
coset $H \cdot g$.  Our inverse chain homotopy equivalence is then given by the
maps $j \co \CC^n(H;\Q) \rightarrow \CC^n(G,H;\Q)$ defined by
$$j(\phi)(g_0,\ldots,g_n) = \phi(g_0 (\hat{g}_0)^{-1}, \ldots, g_n (\hat{g}_n)^{-1}) \quad \quad (g_0,\ldots,g_n \in G).$$
This formula makes sense since $g (\hat{g})^{-1} \in H$ for all $g \in G$.  
It is easy to see that $j(\phi)$ is $H$-equivariant and compatible with the coboundary maps.

\begin{remark}
If $H$ is normal in $G$, then the diagonal action of $G$ on $\CC^n(G,H;\Q)$ descends to an action of $G/H$.
By definition, the cochains $\CC^n(G;\Q)$ are simply the invariants $\CC^n(G,H;\Q)^{G/H}$, that is, the subset of cochains that are fixed by $G/H$.
If $G/H$ is finite, then an invariant element can be found simply by taking the sum (or average) of the $G/H$--orbit of any given element of $\CC^n(G,H;\Q)$. 
The procedure of taking a cocycle in $\CC^n(H;\Q)$, representing its cohomology class with a cocycle in $\CC^n(G,H;\Q)$, and summing its orbit to get a 
cocycle in $\CC^n(G,\Q)$ defines a map $\HH^n(H;\Q)\to \HH^n(G;\Q)$ called the \emph{transfer map}.
Our construction below builds a cocycle formally using the same procedure that defines the transfer map; however, our subgroup is not of finite index so there is no transfer map to define.
As we will see below, our construction works because we use cocycles that satisfy a kind of local finiteness condition.
\end{remark}

\subsection{The cocycles $\zeta_r$}
\label{section:zeta}

Recall that we have fixed some $y \in Y$ and some $a,b \in X \cup Z$ such that $a \neq b$.  Let
$A$ and $B$ be the images of $a$ and $b$ in $\overline{V}/\overline{V}_Y$, respectively.  Lemma
\ref{lemma:yvbasis} says that $\{\text{$y^v$ $|$ $v \in \overline{V} / \overline{V}_Y$}\}$ is
a basis for the free abelian group $N_Y/M$.  For $r \in \Z$, let $\alpha_r' \co N_Y/M \rightarrow \Z$
be the homomorphism such that
$$\alpha_r'(y^v) = 
\begin{cases}
1 & \text{if $v = r A$}\\
0 & \text{otherwise}
\end{cases}
\quad \quad (v \in \overline{V}/\overline{V}_Y).$$
and let $\alpha_r \co \LKernkl \rightarrow \Z$ be the composition
$$\LKernkl \stackrel{I_b}{\longrightarrow} N_Y/M \stackrel{\alpha_r'}{\longrightarrow} \Z.$$
We thus have elements $[\alpha_r] \in \HH^1(\LKernkl;\Q)$.  By the formulas in \S \ref{section:groupchains},
the cohomology class $[\alpha_r] \cup [\alpha_0] \in \HH^2(\LKernkl;\Q)$ can be represented by
the cocycle $\eta_r \in C^2(\LKernkl;\Q)$ defined by the formula
$$\eta_r(g_0,g_1,g_2) = (\alpha_r(g_1) - \alpha_r(g_0))(\alpha_0(g_2) - \alpha_0(g_1))
\quad \quad (g_0,g_1,g_2 \in \LKernkl).$$

Our goal is to modify $\eta_r$ so that it can be extended to a cocycle on $\BKernkl$.
Using the recipe in \S \ref{section:groupchains}, we first extend $\eta_r$ to a cocycle 
$\kappa_r \in C^2(\BKernkl,\LKernkl;\Q)$ using a chain homotopy equivalence 
defined in terms of a set of coset representatives.  We will need to use a special
set of coset representatives which we now describe.
It is not hard to see that $J'_y\co \BKernkl\to  \overline{V}/\overline{V}_Y$ is surjective.
Then by the definition of $\LKernkl$ as $\Ker(J'_y)$, we have a short exact sequence:
\[1\longrightarrow \LKernkl \longrightarrow \BKernkl\stackrel{J'_y}{\longrightarrow} \overline{V}/\overline{V}_Y\longrightarrow 1.\]
We choose a normalized section $\sigma$ to $J'_y$, so that $\sigma\co\overline{V}/\overline{V}_Y\to \BKernkl$ is a set map that is a right-inverse to $J'_y$ and such that $\sigma(1)=1$.
Now, the set
$$S = \{\text{$C_{y,s}$ $|$ $s \in X \cup Z$}\} \subset \BKernkl$$
projects under $J'_y$ to a basis for the free abelian group $\overline{V}/\overline{V}_Y$.  We
can therefore choose $\sigma(x) \in \BKernkl$ such that $\sigma(x)$ is contained in
the subgroup generated by $S$.  Next, for $g \in \BKernkl$ let $\hat{g} = \sigma(J'_y(g)) \in\BKernkl$.  
So by construction, for any $g$ in $\BKernkl$, the element $\hat{g}$ 
is a representative of the coset $\LKernkl \cdot g\cdot$.
Our formula for $\kappa_r$ is then
\begin{align*}
\kappa_r(g_0,g_1,g_2) &= \eta_r(g_0 \hat{g}_0^{-1}, g_1 \hat{g}_1^{-1}, g_2 \hat{g}_2^{-1})\\
&= (\alpha_r(g_1 \hat{g}_1^{-1}) - \alpha_r(g_0 \hat{g}_0^{-1}))(\alpha_0(g_2 \hat{g}_2^{-1}) - \alpha_0(g_1 \hat{g}_1^{-1})).
\end{align*}

Now, if $\kappa_r \in C^2(\BKernkl,\LKernkl;\Q)$ was an element of 
$C^2(\BKernkl;\Q) \subset C^2(\BKernkl,\LKernkl;\Q)$, then it would define
a cohomology class in $\HH^2(\BKernkl;\Q)$.  
Alas, this is not true.  
However, we will be able to average out the action of $\overline{V}/\overline{V}_Y$ on $\kappa_r$ (as hinted in Remark~\ref{re:average}) to get a cocycle that is honestly an element of $C^2(\BKernkl;\Q)$.

We now consider that action.
The diagonal left action of $\BKernkl$ on $C^2(\BKernkl,\LKernkl;\Q)$ descends to an action of $\overline{V}/\overline{V}_Y$, because the cochains in this group are already $\LKernkl$--invariant.
Specifically, for 
$x \in \overline{V}/\overline{V}_Y$, the action is given by
$$(x\cdot \kappa_{r})(g_0,g_1,g_2) = \kappa_r(\sigma(x)^{-1} g_0, \sigma(x)^{-1} g_1, \sigma(x)^{-1} g_2)
\quad \quad (g_0,g_1,g_2 \in \BKernkl).$$

We can unravel our expression for $x\cdot \kappa_r$ by considering an appropriate action of $\overline{V}/\overline{V}_Y$ on $\Hom(\LKernkl,\Z)$.
For $\beta\in\Hom(\LKernkl,\Z)$ and $x\in\overline{V}/\overline{V}_Y$, define
\[(x\cdot \beta)(g)=\beta(\sigma(x)^{-1}g\sigma(x)).\]
This is clearly a well-defined left action: $\BKernkl$ acts on $\Hom(\LKernkl,\Z)$ by conjugating inputs, and since the action of conjugation by $\LKernkl$ on $\Hom(\LKernkl,\Z)$ is trivial, the action descends.
Now we can better explain $x\cdot\kappa_r$ with the following lemma.

\begin{lemma}
\label{lemma:kapparxformula}
For all $g_0,g_1,g_2 \in \BKernkl$, we have
$$(x\cdot\kappa_{r})(g_0,g_1,g_2) = ((x\cdot\alpha_{r})(g_1 \hat{g}_1^{-1}) - (x\cdot\alpha_{r})(g_0 \hat{g}_0^{-1}))((x\cdot\alpha_{0})(g_2 \hat{g}_2^{-1}) - (x\cdot\alpha_{0})(g_1 \hat{g}_1^{-1})).$$
\end{lemma}
\begin{proof}
From the definitions, it is enough to show that for any $r\in \Z$ and $g\in\BKernkl$, we have
$\alpha_r(h\hat{h}^{-1})=(x\cdot\alpha_r)(g \hat{g}^{-1})$, where $h=\sigma(x)^{-1}g$.
As noted before, $S$ projects to a basis for the free abelian
group $\overline{V}/\overline{V}_Y$.
By the definition of $\sigma$, we know $\hat{h}^{-1}$ and $\hat{g}^{-1}\sigma(x)$ are in $\Span{S}$
and map to the same element of $\overline{V}/\overline{V}_Y$ under $J'_y$.
Since $J'_y$ maps $S$ to a basis for its image (a free abelian group), the
kernel of $J'_y|_{\Span{S}}$ is $[\Span{S},\Span{S}]$.  Thus
there exists some $w \in [\Span{S},\Span{S}]$ such that $\hat{h}^{-1} =  \hat{g}^{-1}\sigma(x)w$.
Of course $\alpha_{k}(w) = 0$.  
Then
\[
\alpha_r(h\cdot\hat{h}^{-1})
=\alpha_r(\sigma(x)^{-1}g\hat{g}^{-1}\sigma(x)w)=\alpha_r(\sigma(x)^{-1}g\hat{g}^{-1}\sigma(x))
=(x\cdot\alpha_r)(g \hat{g}^{-1}),
\]
as desired.
\end{proof}

This brings us to the following important lemma.

\begin{lemma}
\label{lemma:finitelymany}
Fix $r \in \Z$ and $g \in \LKernkl$.  There then only exist finitely many 
$x \in \overline{V}/\overline{V}_Y$ such that $(x\cdot \alpha_{r})(g) \neq 0$.
\end{lemma}
\begin{proof}
Write
$$I_b(g) = y^{v_1} + \cdots + y^{v_m} \quad \quad (v_i \in \overline{V}/\overline{V}_Y).$$
For $x \in \overline{V}/\overline{V}_Y$, we then have
$$I_b(\sigma(x)^{-1} g \sigma(x)) = y^{v_1+x} + \cdots + y^{v_m+x}.$$
It follows that $(x\cdot\alpha_{r})(g) = 0$ unless $x = r A - v_i$ for some $1 \leq i \leq m$.
\end{proof}

We now define $\zeta_r \co (\BKernkl)^3 \rightarrow \Q$ by the formula
$$\zeta_r(g_0, g_1, g_2) = \sum_{x \in \overline{V}/\overline{V}_Y} (x\cdot\kappa_{r})(g_0, g_1, g_2) \quad \quad (g_0,g_1,g_2 \in \BKernkl).$$
A priori this infinite sum does not make sense; however, Lemmas \ref{lemma:kapparxformula}
and \ref{lemma:finitelymany} imply that only finitely many terms of it are nonzero for any
particular choice of $g_i$.  Since $x\cdot\kappa_{r}$ is a cocycle in $C^2(\BKernkl, \LKernkl;\Q)$ for
all $x \in \overline{V}/\overline{V}_Y$, we have that $\zeta_r$ is a cocycle in 
$C^2(\BKernkl, \LKernkl;\Q)$.  Moreover,  by construction we have
$\zeta_r(\sigma(x)^{-1} g_0, \sigma(x)^{-1}g_1,\sigma(x)^{-1} g_2) = \zeta_r(g_0,g_1,g_2)$ for all
$x \in \overline{V}/\overline{V}_Y$ and $g_0,g_1,g_2 \in \BKernkl$.  This implies
the following lemma.

\begin{lemma}
For all $r \in \Z$, we have that $\zeta_r$ is a cocycle in $C^2(\BKernkl;\Q)$.  Hence $\zeta_r$ defines
a cohomology class $[\zeta_r] \in \HH^2(\BKernkl;\Q)$.
\end{lemma}

\subsection{Evaluating $[\zeta_r]$ on $\eta_m$:  the proof of Lemma \ref{lemma:zeta}}
\label{section:evaluatezeta}

We now conclude this section by proving Lemma \ref{lemma:zeta}.  Let us
recall the setup.  For $m \in \Z$, define $f_m=C_{y,a}^m C_{b,y} C_{y,a}^{-m}$
and $g = C_{a,y} C_{b,y}$.  It is easily verified that $f_m$ and $g$ commute and hence
determine an abelian cycle $\eta_m \in \HH_2(\BKernkl;\Q)$.  Next, recall
that there is an evaluation pairing
$$\omega \co \HH^2(\BKernkl;\Q) \times \HH_2(\BKernkl;\Q) \longrightarrow \Q.$$
We must prove that for $r,m \geq 1$, we have
$$\omega([\zeta_r], \eta_m) =
\begin{cases}
2 & \text{if $r=m$}\\
0 & \text{otherwise}
\end{cases}$$

Observe first that $f_m, g \in \LKernkl$.  Using the formula from \S \ref{section:groupchains}, we
see that
$$\omega([\zeta_r],\eta_m) = \sum_{x \in \overline{V}/\overline{V}_Y} \bigl((x\cdot \alpha_{r})(f_m)\cdot(x\cdot \alpha_{0})(g) - (x\cdot\alpha_{r})(g)\cdot(x\cdot\alpha_{0})(f_m)\bigr).$$
We must evaluate the various terms in this expression.  Observe that $I_b(f_m)$ is the image of
$$f_m(b) b^{-1} = (a^m y a^{-m} b a^m y^{-1} a^{-m}) b^{-1} = y^{a^m} (y^{b a^m})^{-1}$$
in $\overline{V}/\overline{V}_Y$.  
As usual, $A$ and $B$ denote the images of $a$ and $b$ in
$\overline{V}/\overline{V}_Y$, so we get that $I_b(f_m) = y^{m A} - y^{mA+B}$.  Similarly,
$I_b(g)$ is the image of 
$$g(b) b^{-1} = (y b y^{-1}) b^{-1} = y (y^b)^{-1}$$
in $\overline{V}/\overline{V}_Y$, so $I_b(g) = y - y^B$.  From these two calculations,
we obtain that
$$(x\cdot \alpha_{s})(f_m) =
\begin{cases}
1 & \text{if $x = (s-m)A$}\\
-1 & \text{if $x = (s-m)A-B$}\\
0 & \text{otherwise}
\end{cases}
\quad \quad \text{and} \quad \quad
(x\cdot \alpha_{s})(g) =
\begin{cases}
1 & \text{if $x = sA$}\\
-1 & \text{if $x = sA-B$}\\
0 & \text{otherwise}
\end{cases}$$
for all $s \in \Z$ and $x \in \overline{V}/\overline{V}_Y$.  This implies that
\begin{align*}
(x\cdot\alpha_{r})(f_m)\cdot(x\cdot \alpha_{0})(g)&= 
\begin{cases}
1 & \text{if $r=m$ and either $x=0$ or $x=-B$}\\
0 & \text{otherwise}
\end{cases}\\
(x\cdot\alpha_{r})(g)\cdot(x\cdot \alpha_{0})(f_m) &= 
\begin{cases}
1 & \text{if $r=-m$ and either $x = rA$ or $x = rA-B$}\\
0 & \text{otherwise}
\end{cases}
\end{align*}
This allows us to conclude that for $r, m \geq 1$, we have
$$\omega(\zeta_r,\eta_m) = \sum_{x \in \overline{V}/\overline{V}_Y} ((x\cdot\alpha_{r})(f_m)\cdot(x\cdot \alpha_{0})
(g) - (x\cdot\alpha_{r})(g)\cdot(x\cdot \alpha_{0})(f_m)) = 
\begin{cases}
2 & \text{if $r=m$}\\
0 & \text{otherwise,}
\end{cases}$$
as desired.

\section{The combinatorial group theory of $\BKer{n}{k}{l}$}
\label{section:generatorsandrelations}

The goal in this section is to prove Theorem \ref{maintheorem:generators}
(which gives generators for $\BKer{n}{k}{l}$) and Theorem \ref{theorem:relations} in 
\S \ref{section:genrelbker} below (which gives relations
for $\BKer{n}{k}{l}$, making precise Theorem \ref{maintheorem:relations}).  We begin in \S \ref{section:jensenwahl} by discussing a presentation of $\AutFB{n}{k}{l}$ due to Jensen--Wahl.
Next, in \S \ref{section:genrelbker} we prove Theorems \ref{maintheorem:generators} and \ref{theorem:relations}.  These proofs will depend on a calculation which is contained in \S \ref{section:action}.

Throughout this section, we will make use of the notation
\[X = \{x_1,\ldots,x_n\} \quad \text{and} \quad Y = \{y_1,\ldots,y_k\} \quad \text{and} \quad Z = \{z_1,\ldots,z_l\}.\]

\subsection{A presentation for $\AutFB{n}{k}{l}$}
\label{section:jensenwahl}

For $1 \leq i,j \leq n$ such that $i \neq j$ and $1 \leq i' \leq n$, let $P_{i,j}$ and $I_{i'}$ be
the elements of $\AutFB{n}{k}{l}$ that have the following behavior for $v \in X \cup Y \cup Z$.
\[P_{i,j}(v) = 
\begin{cases}
x_j & \text{if $v=x_i$}\\
x_i & \text{if $v=x_j$}\\
v & \text{otherwise}
\end{cases}
\quad \text{and} \quad
I_{i'}(v) =
\begin{cases}
x_{i'}^{-1} & \text{if $v=x_{i'}$}\\
v & \text{otherwise}
\end{cases}\]
The automorphisms of the form $P_{i,j}$ are \emph{swaps} and the automorphisms of the form $I_i$ are \emph{inversions}.
Elements of the form $\Mul{w^\epsilon}{v}$ for $w,v\in X\cup Y\cup Z$ are \emph{Nielsen moves}
and elements of the form $\Con{w}{v}$ for $v\in X\cup Y\cup Z$ and $w\in X\cup Z$ are \emph{conjugation moves}.
Throughout the rest of the paper, we will tacitly identify $P_{i,j}$ and $P_{j,i}$.  
We will make an extension of our notation to simplify the statement of certain relations (N2 and Q3) given below: for $v,w\in X\cup Y\cup Z$ and $\epsilon=\pm 1$, we define the symbols
\[\Mul{v^\epsilon}{w^{-1}}=\Mul{v^\epsilon}{w}^{-1}\quad \text{and}\quad \Con{v}{w^{-1}}=\Con{v}{w}^{-1}.\]

The following theorem is a restatement of Nielsen's classical presentation~\cite{NielsenPresentation} for $\Aut F_n$ in our notation and using our composition convention (right to left).
A good reference for Nielsen's presentation is its verification by McCool~\cite{McCoolVerification}.
We use McCool's terminology \hbox{N1-5} for the five classes of relations in Nielsen's presentation.
We restrict our attention to $\Aut(F(X))=\AutFB{n}{0}{0}$.

\begin{theorem}[Nielsen~\cite{NielsenPresentation}]
\label{theorem:autfnpresentation}
The group $\Aut(F(X))$ has the presentation $\GroupPres{S_N}{R_N}$, where 
\begin{align*}
S_N = &\Set{$P_{a,b}$}{$1 \leq a,b \leq n$, $a \neq b$} \cup \Set{$I_a$}{$1 \leq a \leq n$} \\
&\quad \quad \quad \quad \quad \cup \Set{$\Mul{x_a^{\epsilon}}{x_b}$}{$1 \leq a,b \leq n$, $a \neq b$, $\epsilon = \pm 1$},
\end{align*}
and $R_N$ is given in Table~\ref{table:autfnpresentation}.
\end{theorem}

\begin{table}[htp]
\begin{tabular}{p{0.95\textwidth}}
\toprule
\begin{center}
{\bf Nielsen's relations for $\Aut(F(X))$}
\end{center}
The relations $R_N$ consist of the following, where $\epsilon,\delta\in\{1,-1\}$ and indices $a,b,c,d$ are assumed to be distinct elements of $\{1,\dotsc,n\}$ unless stated otherwise:
\begin{itemize}
\item[N1.\ ] relations for the subgroup generated by inversions and swaps, a signed permutation group:
\begin{itemize}
\item $I_a^2=1$ and $[I_a,I_b]=1$,
\item 
$P_{a,b}^2=1$, $[P_{a,b},P_{c,d}]=1$, and
$P_{a,b}P_{b,c}P^{-1}_{a,b}=P_{a,c}$,
\item $P_{a,b}I_aP^{-1}_{a,b}=I_b$ and $[P_{a,b},I_c]=1$;
\end{itemize}
\item[N2.\ ] relations for conjugating Nielsen moves by inversions and swaps, coming from the natural action of inversions and swaps on $X^{\pm1}$:
\begin{itemize}
\item $P_{a,b}\Mul{x^\epsilon}{y}P_{a,b}^{-1}=\Mul{P_{a,b}(x^\epsilon)}{P_{a,b}(y)}$ for $x,y\in X$,
\item $I_{a}\Mul{x^\epsilon}{y}I_{a}^{-1}=\Mul{I_a(x^\epsilon)}{I_a(y)}$ for $x,y\in X$;
\end{itemize}
\item[N3.\ ] 
$\Mul{x_a^{-1}}{x_b}^{-1}\Mul{x_b^{-1}}{x_a}\Mul{x_a}{x_b}=I_aP_{a,b}$ and
$\Mul{x_a^{-1}}{x_b}\Mul{x_b}{x_a}\Mul{x_a}{x_b}^{-1}=I_bP_{a,b}$;
\item[N4.\ ] $[\Mul{x_a^\epsilon}{x_b},\Mul{x_c^\delta}{x_d}]=1$ with $a,b,c,d$ not necessarily all distinct, such that $a\neq b$, $c\neq d$, $x_a^\epsilon\notin\{ x_c^\delta, x_d, x_d^{-1}\}$ and $x_c^\delta\notin\{ x_b, x_b^{-1}\}$; and
\item[N5.\ ] $\Mul{x_b^\epsilon}{x_a}\Mul{x_c^\delta}{x_b}^{\epsilon}=\Mul{x_c^\delta}{x_b}^{\epsilon}\Mul{x_b^\epsilon}{x_a}\Mul{x_c^\delta}{x_a}$.
\end{itemize}\\
\bottomrule
\end{tabular}
\caption{
}
\label{table:autfnpresentation}
\end{table}

We now turn to $\AutFB{n}{k}{l}$.  The following presentation is due to Jensen-Wahl \cite{JensenWahl}, who
made use of an algorithm for generating such a presentation due to McCool \cite{McCoolPresentation}.

\begin{theorem}[{Jensen-Wahl, \cite{JensenWahl}}]
\label{theorem:jensenwahl}
The group $\AutFB{n}{k}{l}$ has a finite presentation $\GroupPres{S}{R}$ where 
the generating set $S$ consists of the following elements:
\begin{itemize}
\item $P_{i,j}$ for $1 \leq i,j \leq n$ such that $i \neq j$,
\item $I_{i}$ for $1 \leq i \leq n$,
\item $\Mul{x^{\epsilon}}{v}$ for $\epsilon = \pm 1$, $x \in X$, and $v \in X \cup Y \cup Z$ such that $v \neq x$,
\item $\Con{v}{w}$ for $v \in Y \cup Z$ and $w \in X \cup Y \cup Z$ such that $v \neq w$,
\end{itemize}
and the relations $R$ are given in Table~\ref{table:jensenwahl}.
\end{theorem}

\begin{table}[htp]
\begin{tabular}{p{0.95\textwidth}}
\toprule
\begin{center}
{\bf Jensen and Wahl's relations for $\AutFB{n}{k}{l}$}
\end{center}

The relations $R$ consist of the following:
\begin{itemize}
\item[Q1.\ ] all the relations $R_N$ for $\Aut(F(X))$ from Nielsen's presentation, which are relations between among elements of $S_N$, a subset of $S$;
\item[Q2.\ ]
relations that certain generators commute:
\begin{itemize}
\item[(1)\ ] $[\Mul{x^\epsilon}{v},\Mul{w^\delta}{z}]=1$ for $\epsilon,\delta=\pm1$, $x,v,w\in X$ and $z\in Y\cup Z$, with $x^\epsilon\neq w^\delta$ and $x \neq v$;
\item[(2)\ ] $[\Mul{x^\epsilon}{v},\Con{w}{z}]=1$ for $\epsilon=\pm1$, $x,v\in X$ and $w,z\in Y\cup Z$, with $w\notin\{v,z\}$ and $x \neq v$;
\item[(3)\ ] $[\Con{u}{v},\Con{w}{z}]=1$ for $u,v,w,z\in Y\cup Z$ with $u\notin\{v,w,z\}$ and $w\notin\{v,z\}$;
\end{itemize}
\item[Q3.\ ] 
relations conjugating other generators by swaps and inversions:
\begin{itemize}
\item[(1)\ ] $P_{i,j} \Mul{x^{\epsilon}}{z} P_{i,j}^{-1} = \Mul{P_{i,j}(x^{\epsilon})}{z}$ for $\epsilon = \pm 1$, $x \in X$, $z \in Y \cup Z$, and $1 \leq i,j \leq n$
such that $i \neq j$;
\item[(2)\ ] $P_{i,j} \Con{z}{x} P_{i,j}^{-1} = \Con{z}{P_{i,j}(x)}$ for $x \in X$, $z \in Y \cup Z$, and $1 \leq i,j \leq n$ such that $i \neq j$;
\item[(3)\ ] $I_{i} \Mul{x^{\epsilon}}{z} I_{i}^{-1} = \Mul{I_{i}(x^{\epsilon})}{z}$ for $\epsilon = \pm 1$, $x \in X$, $z \in Y \cup Z$, and $1 \leq i \leq n$;
\item[(4)\ ] $I_{i} \Con{z}{x} I_{i}^{-1} = \Con{z}{I_{i}(x)}$ for $x \in X$, $z \in Y \cup Z$, and $1 \leq i \leq n$;
\end{itemize}
\item[Q4.\ ]
\begin{itemize}
\item[(1)\ ] $\Mul{x^{\epsilon}}{w}^{-\delta} \Mul{w^{\delta}}{v} \Mul{x^{\epsilon}}{w}^{\delta} = \Mul{x^{\epsilon}}{v} \Mul{w^{\delta}}{v}$ for
$\epsilon,\delta=\pm 1$, $x,w \in X$, and $v \in X \cup Y \cup Z$ such that $x \neq w$, $x \neq v$, and $w \neq v$;
\item[(1$'$)\ ] $\Con{z}{x}^{-\epsilon} \Mul{x^{\epsilon}}{v} \Con{z}{x}^{\epsilon} =  \Con{z}{v}\Mul{x^{\epsilon}}{v}$ for $\epsilon = \pm 1$, $x \in X$, $z \in Y \cup Z$,
and $v \in X \cup Y \cup Z$ such that $x \neq v$ and $z \neq v$;
\item[(2)\ ] $\Con{z}{v}^{\epsilon} \Mul{x^{\delta}}{z} \Con{z}{v}^{-\epsilon} = \Mul{x^{\delta}}{v}^{-\epsilon} \Mul{x^{\delta}}{z} \Mul{x^{\delta}}{v}^{\epsilon}$
for $\epsilon, \delta = \pm 1$, $x \in X$, $z \in Y \cup Z$, and $v \in X \cup Y \cup Z$ such that $x \neq v$ and $z \neq v$;
\item[(2$'$)\ ] $\Con{z}{v}^{\epsilon} \Con{w}{z} \Con{z}{v}^{-\epsilon} = \Con{w}{v}^{-\epsilon} \Con{w}{z} \Con{w}{v}^{\epsilon}$ for $\epsilon = \pm 1$,
$z,w \in Y \cup Z$, and $v \in X \cup Y \cup Z$ such that $z \neq w$, $z \neq v$, and $w \neq v$; and
\end{itemize}
\item[Q5.\ ] $\Con{v}{x}^{-\epsilon} \Mul{x^{-\epsilon}}{v} \Con{v}{x}^{\epsilon} = \Mul{x^{\epsilon}}{v}^{-1}$ for $\epsilon = \pm 1$, $x \in X$, and $v \in Y \cup Z$.
\end{itemize}\\
\bottomrule
\end{tabular}
\caption{}
\label{table:jensenwahl}
\end{table}

\begin{remark}
In Jensen--Wahl's original version of the presentation, the relations Q1 were not given explicitly; rather they specified that Q1 should consist of some set of relations for $\Aut(F(X))$ with respect to the generating set $S_N$.
Of course, Nielsen's relations suffice.
We have made minor changes to the presentation that are necessary for giving it in our notation and with our conventions;
this includes renumbering the subclasses of relations Q3 and Q4.
\end{remark}

\begin{remark}
In each of the classes of relation Q3.1--4, there are some relations that state that Nielsen moves and conjugation moves not in $S_N$ commute with swaps and inversions that share no indices in common.
These relations do not appear in Jensen--Wahl's original presentation; we believe their omission to be an error in~\cite{JensenWahl}, although a largely inconsequential one.
A careful reading of Jensen--Wahl's proof of this theorem indicates the need for these extra relations.
\end{remark}

\subsection{Generators and relations for $\BKer{n}{k}{l}$}
\label{section:genrelbker}

Our main tool for studying the combinatorial group theory of $\BKer{n}{k}{l}$ is a certain
special presentation for $\AutFB{n}{k}{l}$ which is a slight modification of Jensen-Wahl's presentation
from Theorem \ref{theorem:jensenwahl}.
Our presentation is inspired by the split exact sequence
\[1 \longrightarrow \BKer{n}{k}{l} \longrightarrow \AutFB{n}{k}{l} \longrightarrow \AutFB{n}{0}{l} \longrightarrow 1\]
from \S \ref{section:introduction}.
Since this sequence is split, $\AutFB{n}{k}{l}$ is the semidirect product of 
$\AutFB{n}{0}{l}$ and $\BKer{n}{k}{l}$.  Our presentation resembles the standard presentation
for a semidirect product.

We begin by discussing our generating set, which consists of $S_Q \cup S_K$ where $S_Q$ and $S_K$ are defined by
\begin{align*}
S_Q = &\Set{$P_{i,j}$}{$1 \leq i,j \leq n$, $i \neq j$} \cup \Set{$I_{i}$}{$1 \leq i \leq n$} \cup 
\Set{$\Con{z}{v}$}{$v \in X\cup Z$, $z \in Z$} \\
&\quad\cup 
\Set{$\Mul{x^{\epsilon}}{v}$}{$x \in X$, $v \in X \cup Z$, $x \neq v$, $\epsilon = \pm 1$} \\
\end{align*}
and
\begin{align*}
S_K =
 & \Set{$\Mul{x^{\epsilon}}{y}$}{$x \in X$, $y \in Y$, $\epsilon = \pm 1$}\cup \Set{$\Con{z}{y}$}{$z \in Z$, $y \in Y$}\\
&\quad \cup \Set{$\Con{y}{v}$}{$y \in Y$, $v \in X \cup Y \cup Z$, $y \neq v$}.
\end{align*}
Of course, $S_Q \cup S_K$ is the generating set from Theorem \ref{theorem:jensenwahl}.  
Clearly $S_K \subset \BKer{n}{k}{l}$, and by Theorem \ref{theorem:jensenwahl} the set $S_Q$ generates
$\AutFB{n}{0}{l} \subset \AutFB{n}{k}{l}$.  
Table \ref{table:action} gives an expression in the 
generating set $S_K$ for $s t s^{-1}$ and $s^{-1} t s$, for each choice of $s \in S_Q$ and $t \in S_K$.
The correctness of these expressions as equations in $\AutFB{n}{k}{l}$ can be verified by direct computation, or by using the relations in Theorem~\ref{theorem:jensenwahl}.
We have provided hints for the reader who wishes to derive these relations from the earlier ones.
With Table~\ref{table:action}, we can prove Theorem \ref{maintheorem:generators}, which we recall asserts that $S_K$ generates $\BKer{n}{k}{l}$.

{\renewcommand{\arraystretch}{1.2}
\begin{table}[t!p]
\begin{center}
\begin{tabular}{l@{\hspace{28pt}}l@{\hspace{28pt}}l@{\hspace{28pt}}l}
\toprule
$\boldsymbol{t \in S_K}$ & $\boldsymbol{s \in S_Q^{\pm1}}$ & $\boldsymbol{s t s^{-1}}$ & Relations Used \\
\midrule
$\Mul{x_a^{\epsilon}}{y}$ & $\Mul{x_a^{\epsilon}}{x_b}^\delta$      & $\Con{y}{x_b}^{-\delta} \Mul{x_a^{\epsilon}}{y} \Con{y}{x_b}^{\delta}$ & Q4.2 \\
                            & $\Mul{x_a^{\epsilon}}{z_i}^\delta$    & $\Con{y}{z_i}^{-\delta} \Mul{x_a^{\epsilon}}{y} \Con{y}{z_i}^{\delta}$  & Q4.2\\
			    & $\Mul{x_b^{\delta}}{x_a}^\epsilon$    & $\Mul{x_a^\epsilon}{y} \Con{y}{x_a}^{-\epsilon} \Mul{x_b^{\delta}}{y}^{-1} \Con{y}{x_a}^{\epsilon}$ & Q5, Q2.1-2, Q4.1 \\
 			    & $\Mul{x_b^{\delta}}{x_a}^{-\epsilon}$ & $\Mul{x_a^\epsilon}{y} \Mul{x_b^{\delta}}{y}$ & Q4.1 \\
                            & $\Con{z_i}{x_a}^\epsilon$             &  $\Mul{x_a^\epsilon}{y} \Con{y}{x_a}^{-\epsilon} \Con{z_i}{y}^{-1} \Con{y}{x_a}^{\epsilon}$ & Q5, Q2.2-3, Q4.1$'$ \\
                            & $\Con{z_i}{x_a}^{-\epsilon}$          & $\Mul{x_a^\epsilon}{y} \Con{z_i}{y}$ & Q4.1$'$\\
                            & $P_{a,b}$                             & $\Mul{x_b^{\epsilon}}{y}$ & Q3.1 \\
			    & $I_{a}$                               & $\Mul{x_a^{-\epsilon}}{y}$ & Q3.3 \\
\midrule
$\Con{y}{x_a}$              & $\Mul{x_a^\epsilon}{x_b}^\delta$      & $( \Con{y}{x_a}^\epsilon \Con{y}{x_b}^\delta)^\epsilon$ & Q4.1$'$ \\
                            & $\Mul{x_a^\epsilon}{z_i}^\delta$      & $( \Con{y}{x_a}^\epsilon \Con{y}{z_i}^\delta)^\epsilon$ & Q4.1$'$ \\
                            & $P_{a,b}$                             & $\Con{y}{x_b}$     & Q3.2  \\
			    & $I_{a}$                               & $\Con{y}{x_a}^{-1}$ & Q3.4 \\
\midrule
$\Con{z_i}{y}$              & $\Con{z_i}{x_a}^\epsilon$             & $\Con{y}{x_a}^{-\epsilon} \Con{z_i}{y} \Con{y}{x_a}^\epsilon$ & Q4.2$'$ \\
                            & $\Mul{x_a^{\epsilon}}{z_i}^\delta$    & $\Con{z_i}{y}\big[\Mul{x_a^{\epsilon}}{y},\Con{y}{z_i}^{-\delta}\big]$ & Q4.2 (twice) \\
			    & $\Con{z_i}{z_j}^\epsilon$             & $\Con{y}{z_j}^{-\epsilon} \Con{z_i}{y} \Con{y}{z_j}^\epsilon$ & Q4.2$'$ \\
			    & $\Con{z_j}{z_i}^\epsilon$             & $\Con{z_i}{y}\big[\Con{z_j}{y},\Con{y}{z_i}^{-\epsilon}\big] $   & Q4.2$'$ (twice) \\
\midrule
$\Con{y}{z_i}$              & $\Con{z_i}{x_a}^\epsilon$             & $\Con{y}{x_a}^{-\epsilon} \Con{y}{z_i} \Con{y}{x_a}^\epsilon$ & Q4.2$'$ \\
                            & $\Con{z_i}{z_j}^\epsilon$             & $\Con{y}{z_j}^{-\epsilon} \Con{y}{z_i} \Con{y}{z_j}^\epsilon$ & Q4.2$'$ \\
\bottomrule
\end{tabular}
\caption{
For $s \in S_Q^{\pm1}$ and $t \in S_K$, this table gives $s t s^{-1}$ in terms of the generating set $S_K$.
If there is no entry for a particular $s \in S_Q^{\pm}$ and $t \in S_K$, then $s t s^{-1} = s^{-1} t s = t$ (and these are relations in Q2).
In this table, $y\in Y$, $\epsilon,\delta=\pm1$, $1\leq a,b\leq n$ and $1\leq i,j\leq l$.
}
\label{table:action}
\end{center}
\end{table}
}

\begin{proof}[Proof of Theorem \ref{maintheorem:generators}]
In the presentation for $\AutFB{n}{k}{l}$ from Theorem \ref{theorem:jensenwahl}, setting
$t=1$ for all $t \in S_K$ results in the presentation from the same theorem for $\AutFB{n}{0}{l}$.  This
implies that $S_K$ normally generates $\BKer{n}{k}{l}$.  The calculations in Table \ref{table:action}
then show that the subgroup of $\AutFB{n}{k}{l}$ generated by $S_K$ is normal, so $S_K$ generates $\BKer{n}{k}{l}$.
\end{proof}

We now turn to the relations for our new presentation of $\AutFB{n}{k}{l}$, which consist of $R_Q \cup R_K \cup R_{\conj}$, where $R_Q$, $R_K$, and
$R_{\conj}$ are as defined below. 
First, $R_Q \subset F(S_Q)$ is the subset of the relations $R$ from Theorem~\ref{theorem:jensenwahl} that are in $F(S_Q)$.
By inspecting Theorem~\ref{theorem:jensenwahl}, it is apparent that $\GroupPres{S_Q}{R_Q} \cong \AutFB{n}{0}{l}$, since these are exactly the relations that appear when $k=0$.
Next, let 
\begin{equation}
\label{eqn:definef}
f \colon S_Q^{\pm 1} \times S_K \rightarrow F(S_K)
\end{equation}
be the function that takes $(s^{\epsilon},t) \in S_Q^{\pm 1} \times S_K$ to the expression for 
$s^{\epsilon} t s^{-\epsilon}$ given by Table \ref{table:action}.  
Define
\[R_{\conj} = \Set{$s^{\epsilon} t s^{-\epsilon} (f(s^{\epsilon},t))^{-1}$}{$s \in S_Q$, $t \in S_K$, $\epsilon = \pm 1$}.\]
Finally, we define $R_K$ to be the subset of the relations $R$ from Theorem~\ref{theorem:jensenwahl} that lie entirely within $F(S_K)$.
The reader can check that this coincides with the set of relations R1-R5 from \S \ref{section:introduction}.

We can now state our presentation.

\begin{proposition}
\label{proposition:autfnklpres}
We have $\AutFB{n}{k}{l} \cong \GroupPres{S_K \cup S_Q}{R_K \cup R_Q \cup R_{\conj}}$.
\end{proposition}

Before proving this, we record a few additional relations.
The relations C2 are used to prove Proposition~\ref{proposition:autfnklpres}; the relations C1 will be used later.
\begin{lemma}\label{lemma:newcommutators}
The following relations in $S_K$ follow from $R_K$ and $R_{\conj}$, and therefore hold in $\GroupPres{S_K \cup S_Q}{R_K \cup R_Q \cup R_{\conj}}$:
\begin{itemize}
\item[C1.\ ]
\begin{itemize}
\item[(1)] $\big[\Con{y}{v}^\epsilon\Mul{x_a^\delta}{y}\Con{y}{v}^{-\epsilon},\ \Mul{x_b^\zeta}{y}\big]=1$
for $\epsilon,\delta,\zeta=\pm1$, $x_a,x_b\in X$, $y\in Y$ and $v\in X\cup Z$ with $x_a^\delta\neq x_b^\zeta$, $v\notin\{x_a,x_b,y\}$.
\item[(2)] $\big[\Con{y}{v}^\epsilon\Mul{x^\delta}{y}\Con{y}{v}^{-\epsilon},\ \Con{z}{y}\big]=1$
for $\epsilon,\delta=\pm1$, $x\in X$, $y\in Y$, $z\in Y\cup Z$ and $v\in X\cup Z$, with $z\neq y$ and $v\notin\{z,x,y\}$.
\item[(3)] $\big[\Con{y}{v}^\epsilon\Con{z}{y}\Con{y}{v}^{-\epsilon},\ \Con{w}{y}\big]=1$
for $\epsilon=\pm1$, $y\in Y$, $z,w\in Y\cup Z$ and $v\in X\cup Z$ with $w\neq u$, $y\notin\{z,w\}$ and $v\notin \{z,w,y\}$.
\end{itemize}
\item[C2.\ ]
\begin{itemize}
\item[(1)]
$\big[\Con{y}{z}^\epsilon\Mul{x^\delta}{y}\Con{y}{z}^{-\epsilon},\ \Con{z}{y}\Mul{x^\delta}{y}\big]=1$ for $\epsilon,\delta=\pm1$, $y\in Y$, $z\in Z$ and $x\in X$.
\item[(2)]
$\big[\Con{y}{z}^\epsilon\Con{w}{y}\Con{y}{z}^{-\epsilon},\ \Con{z}{y}\Con{w}{y}\big]=1$ for $\epsilon=\pm1$, $y\in Y$, and $w,z\in Z$ with $w\neq z$.
\end{itemize}
\end{itemize}
\end{lemma}

\begin{proof}
To derive C1.1, start with the relation $[\Mul{x_a^\delta}{y},\Mul{x_b^\zeta}{y}]=1$ from R1.1.
We conjugate this relation by $\Mul{x_a^\delta}{v}^{-\epsilon}$ and move the conjugating elements into the commutator.
The relation C1.1 follows by modifying the conjugates of commutator factors using relations from $R_{\conj}$.
The derivations for C1.2 and C1.3 are completely parallel.

To derive C2.1, start with the relation $[\Mul{x^\delta}{y},\Con{v}{y}]=1$ from R1.2.
We conjugate by $\Mul{x^\delta}{v}^{-\epsilon}$, move the conjugating elements into the commutator, and apply relations from $R_{\conj}$ to get 
\[\big[ \Con{y}{v}^\epsilon\Mul{x^\delta}{y}\Con{y}{v}^{-\epsilon},\ \Con{v}{y}\Mul{x^\delta}{y}\Con{y}{v}^\epsilon\Mul{x^\delta}{y}^{-1}\Con{y}{v}^{-\epsilon}\big]=1.\]
This is a product of the form $[g,hg^{-1}]=1$, so of course the corresponding expression $[g,h]=1$ follows, which is relation C2.1.
The derivation of C2.2 is parallel.
\end{proof}

\begin{proof}[Proof of Proposition~\ref{proposition:autfnklpres}]
Since all the relations in $R_K \cup R_Q \cup R_{\conj}$ are relations in $\AutFB{n}{k}{l}$, it
is enough to show that all the relations in Theorem \ref{theorem:jensenwahl} are consequences
of the relations $R_K \cup R_Q \cup R_{\conj}$.  In fact, almost all the relations in
Theorem \ref{theorem:jensenwahl} already lie in $R_K \cup R_Q \cup R_{\conj}$.  As the reader
can easily check, there are three exceptions.
\begin{itemize}
\item Relation Q4.1$'$ when $v \in Y$ and $z \in X \cup Z$.  Recall that this asserts that
\[\Con{v}{x}^{-\epsilon} \Mul{x^{\epsilon}}{z} \Con{v}{x}^{\epsilon} =  \Con{v}{z}\Mul{x^{\epsilon}}{z}.\]
Rearranging the terms, we get
\[\Mul{x^{\epsilon}}{z} \Con{v}{x}^{\epsilon} \Mul{x^{\epsilon}}{z}^{-1} = \Con{v}{x}^{\epsilon} \Con{v}{z},\]
which lies in $R_{\conj}$ (it is written in Table~\ref{table:action} as $\Mul{x^{\epsilon}}{z} \Con{v}{x} \Mul{x^{\epsilon}}{z}^{-1} = (\Con{v}{x}^{\epsilon} \Con{v}{z})^\epsilon$).

\item Relation Q4.2 when $v \in Z$ and $z \in Y$.
This asserts that
\[\Con{v}{z}^{\epsilon} \Mul{x^{\delta}}{v} \Con{v}{z}^{-\epsilon} = \Mul{x^{\delta}}{z}^{-\epsilon} \Mul{x^{\delta}}{v} \Mul{x^{\delta}}{z}^{\epsilon}.\]
Multiplying on the left by $\Mul{x^{\delta}}{v}^{-1}$ and applying two relations from $R_{\conj}$, we
see that we must prove that
\[([\Con{z}{v},\Mul{x^{\delta}}{z}^{-1}] \Con{v}{z})^{\epsilon} \Con{v}{z}^{-\epsilon} = \Con{z}{v} \Mul{x^{\delta}}{z}^{-\epsilon} \Con{z}{v}^{-1} \Mul{x^{\delta}}{z}^{\epsilon}.\]
For $\epsilon = 1$, this is trivially true.
For $\epsilon = -1$, it states that
\[\Con{v}{z}^{-1}\Mul{x^{\delta}}{z}^{-1}\Con{z}{v}\Mul{x^{\delta}}{z}\Con{z}{v}^{-1}\Con{v}{z} = \Con{z}{v} \Mul{x^{\delta}}{z} \Con{z}{v}^{-1} \Mul{x^{\delta}}{z}^{-1}.\]
By relation R1.2, we commute $\Mul{x^{\delta}}{z}^{-1}$ past $\Con{v}{z}^{-1}$ to get
\[(\Mul{x^{\delta}}{z}^{-1}\Con{v}{z}^{-1})(\Con{z}{v}\Mul{x^{\delta}}{z}\Con{z}{v}^{-1})\Con{v}{z} = \Con{z}{v} \Mul{x^{\delta}}{z} \Con{z}{v}^{-1} \Mul{x^{\delta}}{z}^{-1}.\]
By relation C2.1 from Lemma~\ref{lemma:newcommutators}, we can commute $\Mul{x^{\delta}}{z}^{-1}\Con{v}{z}^{-1}$ past $\Con{z}{v}\Mul{x^{\delta}}{z}\Con{z}{v}^{-1}$.
Then the resulting relation is obviously true.

\item Relation Q4.2$'$ when $v$ and $w$ are in $Z$ and $z$  is in $Y$.  
This asserts that
\[\Con{v}{z}^{\epsilon} \Con{w}{v} \Con{v}{z}^{-\epsilon} = \Con{w}{z}^{-\epsilon} \Con{w}{v} \Con{w}{z}^{\epsilon}.\]
Multiplying on the left by $\Con{w}{v}^{-1}$ and applying two relations from $R_{\conj}$, we see that
we must prove that
\[([\Con{z}{v},\Con{w}{z}^{-1}] \Con{v}{z})^{\epsilon} \Con{v}{z}^{-\epsilon}
=(\Con{z}{v} \Con{w}{z} \Con{z}{v}^{-1})^{-\epsilon} \Con{w}{z}^{\epsilon}.\]
For $\epsilon=1$, this is clearly true.  
For $\epsilon = -1$, this becomes
\[\Con{v}{z}^{-1}\Con{w}{z}^{-1}\Con{z}{v}\Con{w}{z}\Con{z}{v}^{-1}\Con{v}{z}
=\Con{z}{v} \Con{w}{z} \Con{z}{v}^{-1}\Con{w}{z}^{-1}.\]
Relation R1.3 allows us to commute $\Con{v}{z}^{-1}$ past $\Con{w}{z}^{-1}$, and the relation becomes
\[(\Con{w}{z}^{-1}\Con{v}{z}^{-1})(\Con{z}{v}\Con{w}{z}\Con{z}{v}^{-1})\Con{v}{z}
=\Con{z}{v} \Con{w}{z} \Con{z}{v}^{-1}\Con{w}{z}^{-1}.\]
Then by relation C2.2 from Lemma~\ref{lemma:newcommutators}, we can commute $\Con{w}{z}^{-1}\Con{v}{z}^{-1}$ past $\Con{z}{v}\Con{w}{z}\Con{z}{v}^{-1}$ and the resulting expression is trivially true.
\end{itemize}
\end{proof}

From Proposition \ref{proposition:autfnklpres}, one might be led to think that
$\BKer{n}{k}{l}$ is isomorphic to $\GroupPres{S_K}{R_K}$.  If it were, then the standard presentation
for a semidirect product would yield Proposition \ref{proposition:autfnklpres}.  
However, this is false (recall that we earlier proved Theorem \ref{maintheorem:notfinitepres}, which
asserts that $\BKer{n}{k}{l}$ is not
finitely presented).  The problem is that $\AutFB{n}{0}{l}$ does not act on $\GroupPres{S_K}{R_K}$; the
obvious candidate for an action from Proposition \ref{proposition:autfnklpres} does not preserve
$R_K$.

To fix this, we will have to add some additional relations.  Recall that
above (see \eqref{eqn:definef}) we defined a function $f \colon S_Q^{\pm 1} \times S_K \rightarrow F(S_K)$.
For fixed $s\in S_Q^{\pm1}$, we have $f(s,\cdot)\colon S_K\to F(S_K)$.
Since $F(S_K)$ is a free group, this extends to a homomorphism $f(s,\cdot)\colon F(S_K)\to F(S_K)$.
These maps assemble to a set map $f\colon S_Q^{\pm1}\to\End(F(S_K))$.
A tedious
but elementary calculation shows that for all $t \in S_Q$ and $s \in S_K$, we have
\begin{equation*}
f(t^{-1},f(t,s)) = s \quad \text{and} \quad f(t,f(t^{-1},s)) = s.
\end{equation*}
This lets us deduce two things.  First, $f(S_Q^{\pm 1}) \subset \Aut(F(S_K))$.  Second, $f(t^{-1})=f(t)^{-1}$.
This second fact lets us extend $f$ to a homomorphism $f\colon F(S_Q)\to \Aut(F(S_K))$.  In other words,
$f$ induces an action of $F(S_Q)$ on $F(S_K)$.
We will still write $f$ as a two-variable function $f \colon F(S_Q) \times F(S_K) \rightarrow F(S_K)$.
Finally, we construct our extended set of relations.
\[\hat{R}_K = \Set{$f(w,r)$}{$w \in F_Q$, $r \in R_K$} \subset F(S_K).\]
One should think of $f(w,r)$ here as the ``conjugate'' of the relation $r$ by $w$.
By construction, $F(S_Q)$ acts on $\GroupPres{S_K}{\hat{R}_K}$.  
We will prove the following proposition in \S \ref{section:action}.

\begin{proposition}
\label{proposition:action}
The action of $F(S_Q)$ on $\GroupPres{S_K}{\hat{R}_K}$ descends to an action
of $\AutFB{n}{0}{l}$ on $\GroupPres{S_K}{\hat{R}_K}$.
\end{proposition}

We can now give a presentation for $\BKer{n}{k}{l}$.

\begin{theorem}
\label{theorem:relations}
We have $\BKer{n}{k}{l} \cong \GroupPres{S_K}{\hat{R}_K}$.
\end{theorem}
\begin{proof}
Set $\Gamma = \GroupPres{S_K}{\hat{R}_K}$.  Using Proposition \ref{proposition:action}, we can form the 
semidirect product of $\Gamma$ and $\AutFB{n}{0}{l}$.  This fits into the following commutative diagram
\[\begin{CD}
1 @>>> \Gamma @>>> \Gamma \rtimes \AutFB{n}{0}{l} @>>> \AutFB{n}{0}{l} @>>> 1\\
@.     @VVV        @VVV                             @VV{\cong}V          @.\\
1 @>>> \BKer{n}{k}{l} @>>> \AutFB{n}{k}{l} @>>> \AutFB{n}{0}{l} @>>> 1
\end{CD}\]
Now, $\Gamma \rtimes \AutFB{n}{0}{l}$ has the presentation $\GroupPres{S_K \cup S_Q}{\hat{R}_K \cup R_Q \cup R_{\conj}}$.
However, using the relations $R_{\conj}$ we see that the relations in $\hat{R}_K \setminus R_K$ are redundant.
Using Proposition \ref{proposition:autfnklpres}, we deduce that 
\[\Gamma \rtimes \AutFB{n}{0}{l} \cong \GroupPres{S_K \cup S_Q}{R_K \cup R_Q \cup R_{\conj}} \cong \AutFB{n}{k}{l}.\]
The five lemma then implies that $\Gamma \cong \BKer{n}{k}{l}$, as desired.
\end{proof}

In \cite{Bartholdi}, Bartholdi defined an \emph{$L$-presentation} to be an expression 
\[ \GroupLPres{S}{Q}{\Phi}{R}\]
where $S$ is a set of generators, $\Phi$ is a set of endomorphisms of $F(S)$, and $Q$ and $R$ 
are sets of relations in $F(S)$.
The group presented by the $L$-presentation above is $F(S)/N$, where $N \lhd F(S)$ is the normal
closure of the set
\[Q\cup\bigcup_{\phi\in\Phi^*}\phi(R).\]
Here $\Phi^*$ denotes the closure of $\Phi\cup\{\Id{F(S))}\}$ under composition.
An $L$-presentation is \emph{finite} if $S$, $Q$, $\Phi$ and $R$ are finite, \emph{ascending} if $Q=\varnothing$, and \emph{injective} if $\Phi$ consists of injective endomorphisms.

\begin{corollary}
\label{corollary:finiteLpresentation}
The group $\BKer{n}{k}{l}$ has the following finite $L$-presentation, which is also ascending and injective.
\[\BKer{n}{k}{l}\cong \GroupLPres{S_K}{\varnothing}{
\{f(t)\}_{t\in S_Q^{\pm1}}
}{R_K}.\]
\end{corollary}
\begin{proof}
This is an immediate consequence of Theorem~\ref{theorem:relations} and the definition of $\hat{R}_K$.
\end{proof}

\subsection{Proof of Proposition \ref{proposition:action}}
\label{section:action}
In this subsection, we again use $\Gamma$ to denote $\GroupPres{S_K}{\hat{R}_K}$.
We prove here that $\AutFB{n}{0}{l}$ acts on~$\Gamma$.
Our proof uses the presentation of $\AutFB{n}{0}{l}$ directly; we would prefer a more conceptual proof but we were unable to find one.

\begin{lemma}
The relations C1.1-3 and C2.1-2 from Lemma~\ref{lemma:newcommutators} hold in $\Gamma$.
\end{lemma}

\begin{proof}
Inspecting the proof of Lemma~\ref{lemma:newcommutators}, we see that all relations C1.1-3 and C2.1-2 follow easily from relations in \Set{$f(t,r)$}{$t\in S_Q^{\pm1}, r\in R_K$}, which is a subset of $\hat{R}_K$.
\end{proof}

For each generator $s\in S_Q\cup S_K$, we define two subsets, the \emph{support}, denoted $\supp(s)\subset (X\cup Y\cup Z)^{\pm1}$, and the \emph{multiplier set}, denoted $\mult{s}\subset X\cup Y\cup Z$.
We define these as follows:
\begin{itemize}
\item $\supp(I_a)=\{x_a,x_a^{-1}\}$ and $\mult(I_a)=\{x_a\}$,
\item $\supp(P_{a,b})=\{x_a,x_a^{-1},x_b,x_b^{-1}\}$ and $\mult(P_{a,b})=\{x_a,x_b\}$,
\item $\supp(\Mul{x^\epsilon}{v})=\{x^\epsilon\}$ and $\mult(\Mul{x^\epsilon}{v})=\{v\}$, and
\item $\supp(\Con{w}{v})=\{w,w^{-1}\}$ and $\mult(\Con{w}{v})=\{v\}$.
\end{itemize}
For $s$ with $s^{-1}\in S_K\cup S_Q$, define $\supp(s)=\supp(s^{-1})$ and $\mult(s)=\mult(s^{-1})$.
For $g\in F(S_K\cup S_Q)$ with $g=s_1\dotsm s_m$ (as reduced words) for $s_1,\dotsc,s_m\in (S_K\cup S_Q)^{\pm1}$, define $\supp(g)=\bigcup_{i=1}^m\supp(s_i)$ and $\mult(g)=\bigcup_{i=1}^m\mult(s_i)$.

The following two lemmas can easily be proven using the definition of $f$ and induction on word length.
The proofs are left to the reader.

\begin{lemma}\label{lemma:disjointcommute}
Suppose $s\in F(S_K)$ and $t\in F(S_Q)$.
If $\supp(s)\cap\supp(t)=\varnothing$, $\supp(s)\cap\mult(t)^{\pm1}=\varnothing$ and $\supp(t)\cap\mult(s)^{\pm1}=\varnothing$,
then $f(t,s)=s$ as elements of $F(S_K)$.
\end{lemma}

\begin{lemma}
\label{lemma:actionsupport}
Suppose $s\in F(S_K)$ and $t\in F(S_Q)$.
Then $\supp(f(t,s))\subset \supp(s)\cup\supp(t)\cup\mult(s)^{\pm1}$ and $\mult(f(t,s))\subset\mult(s)\cup\mult(t)$.
\end{lemma}

\begin{proof}[Proof of Proposition~\ref{proposition:action}]
Recall that the goal is to show that $f$ defines an action of $\AutFB{n}{0}{l}\cong \GroupPres{S_Q}{R_Q}$ on $\Gamma$.
From the discussion preceding the statement of Proposition~\ref{proposition:action}, we know that $f$ defines a homomorphism $F(S_Q)\to \Aut(\Gamma)$.
To prove the proposition, 
it is enough to show that each relation in $R_Q$ fixes each generator in $S_K$ as an element of $\Gamma$ under this action.
We do this in several cases, depending on which relations in $R_Q$ we are considering.

\BeginCases
\begin{case}
Relations from Q3, N1 and N2.
\end{case}

In the action of $F(S_Q)$ on $F(S_K)$, these relations fix each element of $S_K$.
This can be verified by direct computation, but it is also easy to see for structural reasons.
The subgroup $\tilde P$ of $F(S_Q)$ generated by swaps and inversions acts on $F(S_Q)$ and $F(S_K)$ directly through the action of a signed permutation group $P$, and on $\Aut(F(S_K))$ diagonally.
This implies that the relations N1 act trivially on $S_K$, because these relations hold in $P$.
The action homomorphism $f\colon F(S_Q)\to\Aut(F(S_K))$ is equivariant with respect to these actions because the definitions in Table~\ref{table:action} are uniform with respect to the subscripts of the generators.
The map $f$ sends the conjugation action of $\tilde P$ on $F(S_Q)$ to the diagonal action on $\Aut(F(S_K))$.
Then all the relations from N2 and Q3 act trivially on $F(S_K)$ because of the equivariance of $f\colon F(S_Q)\to\Aut(F(S_K))$.

\begin{case}\label{case:commute}
Relations from Q2 and N4.
\end{case}
These are exactly the relations in $R_Q$ that state that certain pairs of generators $t_1$, $t_2$ from $S_Q$ commute, where each $t_i$ is a Nielsen move or a conjugation move.
Instead of enumerating the numerous subcases for this case, we treat many cases simultaneously by analyzing the supports and multiplier sets of the generators.

In each case, we can write our relation as $t_1t_2=t_2t_1$ for some $t_1,t_2\in S_Q$.
By inspecting relations Q2 and N4, we see that this implies 
\begin{equation}
\label{equation:commutingrelationsupport}
\supp(t_1)\cap\supp(t_2)=\supp(t_1)\cap\mult(t_2)^{\pm1}=\supp(t_2)\cap\mult(t_1)^{\pm1}=\varnothing.
\end{equation}
By definition, we have $\supp(t_i)\cup\mult(t_i)\subset (X\cup Z)^{\pm1}$ for $i=1,2$.
Suppose $s\in S_K$.
If $\supp(s)\cup\mult(s)\subset Y^{\pm1}$, then both $t_1t_2$ and $t_2t_1$ fix $s$ by Lemma~\ref{lemma:disjointcommute}.
So we suppose this is not the case; however, we must have $\supp(s)\subset Y^{\pm1}$ or $\mult(s)\subset Y$.

If $\supp(s)\subset Y^{\pm1}$, this means $s=\Con{y}{v}$ for some $y\in Y$ and $v\in X\cup Z$.
In particular, this means $(\mult(t_i)^{\pm1}\cup\supp(t_i))\cap\supp(s)=\varnothing$ for $i=1,2$.
If the relation acts nontrivially on $s$, then at least one $t_i$, say $t_1$, satisfies $f(t_1,s)\neq s$.
Then by Lemma~\ref{lemma:disjointcommute}, we have that $\supp(t_1)\subset\{v,v^{-1}\}$.
From the definition of $f$, we know $\supp(f(t_1,s))=\{y,y^{-1}\}$ and $\mult(f(t_1,s))\subset \{v\}\cup\mult(t_1)$.
From Equation~\eqref{equation:commutingrelationsupport}, we know $\supp(t_2)\cap\mult(t_1)^{\pm1}=\varnothing$.
Then since $\mult(t_2)^{\pm1}\cap\supp(f(t_1,s))=\varnothing$, Lemma~\ref{lemma:disjointcommute} implies that either $\supp(t_2)\subset\{v,v^{-1}\}$ or else 
$f(t_2,s)= s$ and $f(t_2t_1,s)= f(t_1,s)$.
Since the latter case implies that $f(t_1t_2,s)= f(t_2t_1,s)$,
we assume that $t_i\subset\{v,v^{-1}\}$ for $i=1,2$.
By Equation~\eqref{equation:commutingrelationsupport}, the only way this is possible is if $s=\Con{y}{x_a}$, $t_1=\Mul{x_a^\epsilon}{y}$ and $t_2=\Mul{x_a^{-\epsilon}}{y}$ for some $x_a^\epsilon\in X^{\pm1}$.
Direct computation shows that $f(t_2t_1,s)=f(t_1t_2,s)$ in $F(S_K)$ in this case.

Now suppose that $\mult(s)\subset Y$.
First we suppose that $\supp(s)\cap\supp(t_1)\neq\varnothing$; in fact, this implies that $\supp(s)=\supp(t_1)$.
Then by Lemma~\ref{lemma:actionsupport}, $\supp(f(t_1,s))\subset \supp(s)\cup\mult(s)^{\pm1}$ and $\mult(f(t_1,s))\subset\mult(t_1)\cup\mult(s)$.
Then $(\supp(t_2)\cup\mult(t_2)^{\pm1})\cap\supp(g)=\varnothing$ and $(\supp(g)\cup\mult(g)^{\pm1})\cap\supp(t_2)=\varnothing$ for $g=s$ and for $g=f(t_1,s)$
(this uses Equation~\eqref{equation:commutingrelationsupport}).
Then $f(t_1t_2,s)=f(t_2t_1,s)$ by Lemma~\ref{lemma:disjointcommute}.
Swapping the roles of $t_1$ and $t_2$, we can now suppose that $\supp(s)\cap\supp(t_i)=\varnothing$ for $i=1,2$.

Still supposing that $\mult(s)\subset Y$, we suppose (as above) that $f(t_1,s)\neq s$.
Then this implies that $\supp(s)\subset\mult(t_1)^{\pm1}$, $\supp(f(t_1,s))\subset \supp(t_1)\cup\supp(s)\cup Y^{\pm1}$
and $\mult(f(t_1,s))\subset\mult(t_1)\cup Y$.
If $\supp(s)\cap\mult(t_2)^{\pm1}=\varnothing$, 
Then $(\supp(t_2)\cup\mult(t_2)^{\pm1})\cap\supp(g)=\varnothing$ and $(\supp(g)\cup\mult(g)^{\pm1})\cap\supp(t_2)=\varnothing$ for $g=s$ and for $g=f(t_1,s)$
implying that $f(t_1t_2,s)=f(t_2t_2,s)$.
So we also assume that $\supp(s)\subset\mult(t_2)^{\pm1}$.
Then there are only nine remaining cases; these are given in Table~\ref{table:commutecheck}.

{\renewcommand{\arraystretch}{1.2}
\begin{table}[h!tp]
\begin{center}
\begin{tabular}{l@{\hspace{28pt}}l@{\hspace{28pt}}l@{\hspace{28pt}}l}
\toprule
$\boldsymbol{s \in S_K}$ & $\boldsymbol{t_1 \in S_Q^{\pm1}}$ & $\boldsymbol{t_2 \in S_Q^{\pm1}}$ & $\boldsymbol{f(t_2t_1,s)^{-1}f(t_1t_2,s)}$ \\
\midrule
$\Mul{x_a}{y}$          & $\Mul{x_b^\epsilon}{x_a}$ & $\Mul{x_c^\delta}{x_a}$ & $\Con{y}{x_a}^{-1}\big[\Mul{x_b^\epsilon}{y},\Mul{x_c^\delta}{y}\big]\Con{y}{x_a}$ \\
                        & $\Mul{x_b^\epsilon}{x_a}$ & $\Con{z_i}{x_a}$       & $\Con{y}{x_a}^{-1}\big[\Mul{x_b^\epsilon}{y},\Con{z_i}{y}\big]\Con{y}{x_a}$ \\
                        & $\Con{z_i}{x_a}$        & $\Con{z_j}{x_a}$       & $\Con{y}{x_a}^{-1}\big[\Con{z_i}{y},\Con{z_j}{y}\big]\Con{y}{x_a}$ \\
\midrule
$\Mul{x_a^{-1}}{y}$     & $\Mul{x_b^\epsilon}{x_a}$ & $\Mul{x_c^\delta}{x_a}$ & $\big[\Mul{x_b^\epsilon}{y}^{-1},\Mul{x_c^\delta}{y}^{-1}\big]$ \\
                        & $\Mul{x_b^\epsilon}{x_a}$ & $\Con{z_i}{x_a}$       & $\big[\Mul{x_b^\epsilon}{y}^{-1},\Con{z_i}{y}^{-1}\big]$ \\
                        & $\Con{z_i}{x_a}$        & $\Con{z_j}{x_a}$       & $\big[\Con{z_i}{y}^{-1},\Con{z_j}{y}^{-1}\big]$ \\
\midrule
$\Con{z_i}{y}$          & $\Mul{x_a^\epsilon}{z_i}^{-1}$      & $\Mul{x_b^\epsilon}{z_i}^{-1}$       & $\Big[\big[\Con{y}{z_i},\Mul{x_a^\epsilon}{y}\big],\big[\Con{y}{z_i},\Mul{x_b^\epsilon}{y}\big]\Big]$ \\
                        & $\Mul{x_a^\epsilon}{z_i}^{-1}$      & $\Con{z_j}{z_i}^{-1}$              & $\Big[\big[\Con{y}{z_i},\Mul{x_a^\epsilon}{y}\big],\big[\Con{y}{z_i},\Con{z_j}{y}\big]\Big]$\\
                        & $\Con{z_j}{z_i}^{-1}$               & $\Con{z_m}{z_i}^{-1}$              & $\Big[\big[\Con{y}{z_i},\Con{z_j}{y}\big],\big[\Con{y}{z_i},\Con{z_m}{y}\big]\Big]$ \\
\bottomrule
\end{tabular}
\caption{
This table includes the computations relevant to Case~\ref{case:commute} of Proposition~\ref{proposition:action}.
In each row, we compute the difference of the actions of the two sides of the relation $t_1t_2=t_2t_1$ on $s$, as an element of $F(S_K)$. 
This includes the only cases where $\mult(s)\in Y$ and inspection of the supports of the generators does not immediately imply that the relation fixes $s$.
}
\label{table:commutecheck}
\end{center}
\end{table}
}

To finish this case, we must show that the expressions for $f(t_2t_1,s)^{-1}f(t_1t_2,s)$ in Table~\ref{table:commutecheck} are always trivial in $\Gamma$.
The entries for $s=\Mul{x_a^\pm1}{y}$ are trivial by relations R1.1-3;
those for $s=\Con{z_i}{y}$ are trickier.
We consider the terms from one of the entries:
\[[\Con{y}{z_i},\Mul{x_a^\epsilon}{y}][\Con{y}{z_i},\Mul{x_b^\delta}{y}]=(\Con{y}{z_i}\Mul{x_a^\epsilon}{y}\Con{y}{z_i}^{-1})\Mul{x_a^\epsilon}{y}^{-1}
(\Con{y}{z_i},\Mul{x_b^\delta}{y}\Con{y}{z_i}^{-1})\Mul{x_b^\delta}{y}^{-1}.\]
By relations from C1.1, the elements $(\Con{y}{z_i}\Mul{x_a^\epsilon}{y}\Con{y}{z_i}^{-1})$ and $\Mul{x_b^\delta}{y}^{-1}$ commute, and the elements $(\Con{y}{z_i},\Mul{x_b^\delta}{y}\Con{y}{z_i}^{-1})$ and $\Mul{x_a^\epsilon}{y}^{-1}$ commute.
By relations from R1.1, the elements $\Mul{x_a^\epsilon}{y}^{-1}$ and  $\Mul{x_b^\delta}{y}^{-1}$ commute, and the elements $(\Con{y}{z_i}\Mul{x_a^\epsilon}{y}\Con{y}{z_i}^{-1})$ and $(\Con{y}{z_i},\Mul{x_b^\delta}{y}\Con{y}{z_i}^{-1})$ commute.
Working in $\Gamma$, this allows us to show:
\[[\Con{y}{z_i},\Mul{x_a^\epsilon}{y}][\Con{y}{z_i},\Mul{x_b^\delta}{y}]=[\Con{y}{z_i},\Mul{x_b^\delta}{y}][\Con{y}{z_i},\Mul{x_a^\epsilon}{y}],\]
meaning that the seventh expression from the last column of the table is trivial in $\Gamma$.
The last two expressions are completely analogous, using R1.2 and C1.2, and R1.3 and C1.3 instead of R1.1 and C1.1, respectively.
\begin{case}\label{case:N3}
Relations from N3.
\end{case}
One of the relations from N3 is of the form $\Mul{x_a^{-1}}{x_b}^{-1}\Mul{x_b^{-1}}{x_a}\Mul{x_a}{x_b}=I_aP_{a,b}$, for $1\leq a,b\leq n$ and $a\neq b$.
We rearrange this as 
\begin{equation*}
\Mul{x_b^{-1}}{x_a}\Mul{x_a}{x_b}=\Mul{x_a^{-1}}{x_b}I_aP_{a,b}.
\end{equation*}
For the rest of this case, let $g=\Mul{x_b^{-1}}{x_a}\Mul{x_a}{x_b}$ and let $h=\Mul{x_a^{-1}}{x_b}I_aP_{a,b}$, so that the relation in question is $g=h$.
Support considerations imply that all elements of $S_K$ are fixed in $F(S_K)$ by this the action of relation except possibly for those of the form $\Mul{x_a}{y}$, $\Mul{x_b}{y}$, $\Con{y}{x_a}$ or $\Con{y}{x_b}$ for $y\in Y$.
Direct computation shows that $f(g,\Con{y}{x_a})=f(h,\Con{y}{x_a})$ and $f(g,\Con{y}{x_b})=f(h,\Con{y}{x_b})$, as elements of $F(S_K)$;
only two possibilities for $s$ are left.
In the case $s=\Mul{x_a}{y}$, we derive $f(g,s)=f(h,s)$ in $\Gamma$:
\begin{align*}
f(g,s)&=\Con{y}{x_b}^{-1}(\Con{y}{x_a}\Mul{x_a}{y}\Con{y}{x_a}^{-1})\Mul{x_b^{-1}}{y}^{-1}\Con{y}{x_b}\quad&&\text{(computation in $F(S_K)$)}\\
&=\Con{y}{x_b}^{-1}\Mul{x_a^{-1}}{y}^{-1}\Mul{x_b^{-1}}{y}^{-1}\Con{y}{x_b} && \text{(relation R5)}\\
&=(\Con{y}{x_b}^{-1}\Mul{x_b^{-1}}{y}^{-1})\Mul{x_a^{-1}}{y}^{-1}\Con{y}{x_b} && \text{(relation R1.1)}\\
&=(\Mul{x_b}{y}\Con{y}{x_b}^{-1})\Mul{x_a^{-1}}{y}^{-1}\Con{y}{x_b} && \text{(relation R5)}\\
&=f(h,s)&&\text{(computation in $F(S_K)$).}
\end{align*}
If $s=\Mul{x_b}{y}$, we compute:
\begin{align*}
f(g,s)&=(\Mul{x_b}{y}\Con{y}{x_b}^{-1}\Mul{x_b^{-1}}{y})\Con{y}{x_a}\Mul{x_a}{y}^{-1}\Con{y}{x_a}^{-1}\Con{y}{x_b} \quad&&\text{(computation in $F(S_K)$)}\\
&=\Con{y}{x_b}^{-1}(\Con{y}{x_a}\Mul{x_a}{y}^{-1}\Con{y}{x_a}^{-1})\Con{y}{x_b} &&\text{(relation R5)}\\
&=\Con{y}{x_b}^{-1}\Mul{x_a^{-1}}{y}\Con{y}{x_b} &&\text{(relation R5)}\\
&=f(h,s) &&\text{(computation in $F(S_K)$).}
\end{align*}
The derivations for the other type of relation in N3 are similar.

\begin{case}\label{case:steinberg}
Relations from N5 and Q4.1.
\end{case}
Any relation in $R_Q$ from N5 or Q4.1 can be written as
\[\Mul{x_c^\delta}{x_b}^\epsilon\Mul{x_b^\epsilon}{v}\Mul{x_c^\delta}{v}=\Mul{x_b^\epsilon}{v}\Mul{x_c^\delta}{x_b}^\epsilon\]
for some $\epsilon,\delta=\pm1$, some $v\in X\cup Z$ and distinct $1\leq b,c\leq n$ with $v\neq x_b,x_c$.
Actually, to write a relation from Q4.1 in this form, we must modify it with a relation from Q1.1; however, since we have already verified that those relations act trivially on $\Gamma$, this makes no difference.
For the rest of this case, let $g_1=\Mul{x_b^\epsilon}{v}\Mul{x_c^\delta}{v}$,  $g_2=\Mul{x_c^\delta}{x_b}^\epsilon$ and $h=\Mul{x_b^\epsilon}{v}\Mul{x_c^\delta}{x_b}^\epsilon$, so that the relation reads $g_2g_1=h$.
Lemma~\ref{lemma:actionsupport} implies that this relation fixes all elements of $S_K$ with the following possible exceptions: $\Mul{v}{y}$, $\Mul{v^{-1}}{y}$ (if $v\in x$), $\Con{v}{y}$ (if $v\in Z$), $\Mul{x_b^\epsilon}{y}$,  $\Mul{x_b^{-\epsilon}}{y}$, $\Mul{x_c^\delta}{y}$,  $\Con{y}{x_b}$ and $\Con{y}{x_c}$.
Further computation shows that $f(g_2g_1,s)=f(h,s)$ in $F(S_K)$ for $s$ equal to $\Mul{v^{-1}}{y}$, $\Mul{x_b^\epsilon}{y}$, $\Mul{x_b^{-\epsilon}}{y}$, $\Mul{x_c^\delta}{y}$, $\Con{y}{x_b}$ or $\Con{y}{x_c}$.
We now check the remaining two possibilities for $s$.
If $v\in X$ and $s=\Mul{v}{y}$, we derive $f(g_2g_1,s)=f(h,s)$ in~$\Gamma$:
\begin{align*}
f(g_2g_1,s)&=f(g_2,\Mul{v}{y}\Con{y}{v}^{-1}\Mul{x_b^\epsilon}{y}^{-1}\Mul{x_c^\delta}{y}^{-1}\Con{y}{v})\quad &&\text{(computation in $F(S_K)$)}\\
&=f(g_2,\Mul{v}{y}\Con{y}{v}^{-1}\Mul{x_c^\delta}{y}^{-1}\Mul{x_b^\epsilon}{y}^{-1}\Con{y}{v}) &&\text{(relation R1.1)}\\
&=\Mul{v}{y}\Con{y}{v}^{-1}\Mul{x_b^\epsilon}{y}^{-1}\Con{y}{v} &&\text{(computation in $F(S_K)$)}\\
&=f(h,s) &&\text{(computation in $F(S_K)$).}
\end{align*}
If $v\in Z$ and $s=\Con{v}{y}$:
\begin{align*}
f(g_2g_1,s)&=f(g_2,\Con{v}{y}\big[\Mul{x_b^\epsilon}{y},\Con{y}{v}^{-1}\big]\big[\Mul{x_c^\delta}{y},\Con{y}{v}^{-1}\big])\quad &&\text{(computation in $F(S_K)$)}\\
&=f(g_2,\Con{v}{y}\big[\Mul{x_b^\epsilon}{y}\Mul{x_c^\delta}{y},\ \Con{y}{v}^{-1}\big]) &&\text{(relations R1.1 and C1.1)}\\
&=\Con{v}{y}\big[\Mul{x_b^\epsilon}{y},\Con{y}{v}^{-1}\big] &&\text{(computation in $F(S_K)$)}\\
&=f(h,s) &&\text{(computation in $F(S_K)$).}
\end{align*}

\begin{case}
Relations from Q4.1$'$.
\end{case}

This case is parallel to Case~\ref{case:steinberg}.
Any relation in $R_Q$ from Q4.1$'$ can be written as
\[\Con{z}{x}^{-\epsilon}\Mul{x^\epsilon}{v}\Con{z}{x}^\epsilon=\Con{z}{v}\Mul{x^\epsilon}{v}\]
for some $\epsilon=\pm1$, $x\in X$, $v\in X\cup Z$ and $z\in Z$ with $v\neq x,z$.
Again, we note that we have modified this relation using a relation from Q1.1, but since these act trivially this is allowable.
By Lemma~\ref{lemma:actionsupport}, the only cases where we might have $f(h^{-1}g_2g_1,s)$ not equal to $s$ are if $s$ equals
$\Mul{v}{y}$, $\Mul{v^{-1}}{y}$ (if $v\in X$), $\Con{v}{y}$ (if $v\in Z$), $\Mul{x^\epsilon}{y}$, $\Mul{x^{-\epsilon}}{y}$, $\Con{y}{x}$, $\Con{z}{y}$ or $\Con{y}{z}$.
Application of the definition of $f$  shows that $f(g_2g_1,s)=f(h,s)$ as elements of $F(S_K)$ unless $s$ is $\Mul{v}{y}$ or $\Con{v}{y}$.
The computations to show that these are fixed in $\Gamma$ by this relation are similar to the corresponding ones in Case~\ref{case:steinberg}.

\begin{case}\label{case:Q4.2}
Relations from Q4.2.
\end{case}

Any relation in $R_Q$ from Q4.2 can be written as
\[\Mul{x^\delta}{z}\Mul{x^\delta}{v}^\epsilon\Con{z}{v}^{\epsilon}=\Mul{x^\delta}{v}^{\epsilon}\Con{z}{v}^{\epsilon}\Mul{x^\delta}{z}\]
for $\delta,\epsilon=\pm1$, $x\in X$, $v\in X\cup Z$ and $z\in Z$ with $v\neq x,z$.
For the rest of this case let $g=\Mul{x^\delta}{z}$ and $h=\Mul{x^\delta}{v}^\epsilon\Con{z}{v}^{\epsilon}$, so that the relation reads $gh=hg$.
By Lemma~\ref{lemma:actionsupport}, the only cases where we might have $f(gh,s)$ not equal to $f(hg,s)$ are where $s$ equals
$\Mul{v^\epsilon}{y}$, $\Mul{v^{-\epsilon}}{y}$ (if $v\in X$), $\Con{v}{y}$ (if $v\in Z$), $\Mul{x^\delta}{y}$, $\Con{y}{x}$, $\Con{z}{y}$ or $\Con{y}{z}$ for some $y\in Y$.
Computing $f$ by the definition in $F(S_K)$ shows that $f(gh,s)=f(hg,s)$ as elements of $F(S_K)$ in all these cases except when $s$ is $\Mul{v^{-\epsilon}}{y}$ or  $\Con{v}{y}$.
If $v\in X$ and $s=\Mul{v^{-\epsilon}}{y}$, we compute:
\begin{align*}
f(gh,s)&=f(g,\Mul{v^{-\epsilon}}{y}\Mul{x^\delta}{y}\Con{z}{y})\quad &&\text{(computation in $F(S_K)$)}\\
&=f(g,\Mul{v^{-\epsilon}}{y}\Con{z}{y}\Mul{x^\delta}{y})&&\text{(relation R1.2)}\\
&=\Mul{v^{-\epsilon}}{y}\Con{z}{y}\Mul{x^\delta}{y} &&\text{(computation in $F(S_K)$)}\\
&=\Mul{v^{-\epsilon}}{y}\Mul{x^\delta}{y}\Con{z}{y} &&\text{(relation R1.2)}\\
&=f(hg,s) &&\text{(computation in $F(S_K)$).}
\end{align*}
If $v\in Z$ and $s=\Con{v}{y}$, we compute:
\begin{align*}
f(gh,s)&=f(g,\Con{v}{y}\big[\Mul{x^\delta}{y},\Con{y}{v}^{-\epsilon}\big]\big[\Con{z}{y},\Con{y}{v}^{-\epsilon}\big])\quad &&\text{(computation in $F(S_K)$)}\\
&=f(g,\Con{v}{y}\big[\Con{z}{y}\Mul{x^\delta}{y},\Con{y}{v}^{-\epsilon}\big])&&\text{(relations C1.2 and R1.2)}\\
&=\Con{v}{y}\big[\Con{z}{y}\Mul{x^\delta}{y},\Con{y}{v}^{-\epsilon}\big] &&\text{(computation in $F(S_K)$)}\\
&=\Con{v}{y}\big[\Mul{x^\delta}{y},\Con{y}{v}^{-\epsilon}\big]\big[\Con{z}{y},\Con{y}{v}^{-\epsilon}\big] &&\text{(relations C1.2 and R1.2)}\\
&=f(hg,s) &&\text{(computation in $F(S_K)$).}
\end{align*}

\begin{case}
Relations from Q4.2$'$.
\end{case}
This case is completely analogous to Case~\ref{case:Q4.2}.
Any relation in $R_Q$ from Q4.2$'$ can be written as:
\[\Con{z_j}{z_i}\Con{z_j}{v}^\epsilon\Con{z_i}{v}^{\epsilon}=\Con{z_j}{v}^\epsilon\Con{z_i}{v}^{\epsilon}\Con{z_j}{z_i}\]
for $\epsilon=\pm1$, $z_i,z_j\in Z$ with $z_i\neq z_j$ and $v\in X\cup Z$ with $v\neq z_i,z_j$.
By Lemma~\ref{lemma:actionsupport}, the only generators in $S_K$ that might not be fixed by this relation are
$\Mul{v^\epsilon}{y}$, $\Mul{v^{-\epsilon}}{y}$ (if $v\in X$), $\Con{v}{y}$ (if $v\in Z$), $\Con{z_i}{y}$, $\Con{y}{z_i}$, $\Con{z_j}{y}$ and $\Con{y}{z_j}$, for $y\in Y$.
Direct computation in $F(S_K)$ shows that these are all fixed except possibly $\Mul{v^{-\epsilon}}{y}$ and $\Con{v}{y}$.
Derivations that show that this relation fixes these elements in $\Gamma$ are parallel to the corresponding derivations 
in Case~\ref{case:Q4.2}, using relations C1.3 and R1.3 instead of C1.2 and R1.2.

\begin{case}
Relations from Q5.
\end{case}
Any relation in $R_Q$ from $Q5$ can be written as
\[\Mul{x^{-\epsilon}}{z}\Con{z}{x}^{\epsilon}=\Con{z}{x}^\epsilon\Mul{x^\epsilon}{z}^{-1}\]
for $\epsilon=\pm 1$, $x\in X$ and $z\in Z$.
The only generators in $S_K$ that might not be fixed by this relation are of the form
$\Mul{x^\epsilon}{y}$, $\Mul{x^{-\epsilon}}{y}$, $\Con{y}{z}$, $\Con{z}{y}$ or $\Con{y}{z}$ for $y\in Y$.
Straightforward computation shows that $\Con{y}{x}$ and $\Con{y}{z}$ are fixed as elements of $F(S_K)$.
Let $g=\Mul{x^{-\epsilon}}{z}\Con{z}{x}^{\epsilon}$ and let $h=\Con{z}{x}^\epsilon\Mul{x^\epsilon}{z}^{-1}$.
If $s=\Mul{x^\epsilon}{y}$:
\begin{align*}
f(g,s)&=(\Mul{x^\epsilon}{y}\Con{y}{x}^{-\epsilon}\Mul{x^{-\epsilon}}{y})\Con{y}{z}\Mul{x^{-\epsilon}}{y}^{-1}\Con{z}{y}^{-1}\Con{y}{z}^{-1}\Con{y}{x}^\epsilon \quad &&\text{(computation in $F(S_K)$)}\\
&=\Con{y}{x}^{-\epsilon}\Con{y}{z}\Mul{x^{-\epsilon}}{y}^{-1}\Con{z}{y}^{-1}\Con{y}{z}^{-1}\Con{y}{x}^\epsilon  &&\text{(relation R5)}\\
&=\Con{y}{x}^{-\epsilon}\Con{y}{z}(\Con{y}{x}^{\epsilon}\Mul{x^{\epsilon}}{y}\Con{y}{x}^{-\epsilon})\Con{z}{y}^{-1}\Con{y}{z}^{-1}\Con{y}{x}^\epsilon &&\text{(relation R5)}\\
&=f(h,s)&&\text{(computation in $F(S_K)$).}\\
\end{align*}
If $s=\Mul{x^{-\epsilon}}{y}$:
\begin{align*}
f(g,s)&=(\Con{y}{z}^{-1}\Mul{x^{-\epsilon}}{y}\Con{y}{z})(\Con{z}{y}\Mul{x^{-\epsilon}}{y})(\Con{y}{z}^{-1}\Mul{x^{-\epsilon}}{y}^{-1}\Con{y}{z})\quad &&\text{(computation in $F(S_K)$)}\\
&= \Con{z}{y}\Mul{x^{-\epsilon}}{y}&&\text{(relation C2.1)}\\
&= \Mul{x^{-\epsilon}}{y}\Con{z}{y}&&\text{(relation R1.2)}\\
&=f(h,s)&&\text{(computation in $F(S_K)$).}\\
\end{align*}
If $s=\Con{z}{y}$:
\begin{align*}
f(h,s)&=\Con{y}{x}^{-\epsilon}\Con{z}{y}\big[(\Con{y}{x}^{\epsilon}\Mul{x^{\epsilon}}{y}\Con{y}{x}^{-\epsilon})\Con{z}{y}^{-1},\ \Con{y}{z}\big]\Con{y}{x}^{\epsilon}\quad &&\text{(computation in $F(S_K)$)}\\
&=\Con{y}{x}^{-\epsilon}\Con{z}{y}\big[\Mul{x^{-\epsilon}}{y}^{-1}\Con{z}{y}^{-1},\ \Con{y}{z}\big]\Con{y}{x}^{\epsilon}&&\text{(relation R5)}\\
&=\Con{y}{x}^{-\epsilon}(\Con{z}{y}\Mul{x^{-\epsilon}}{y}^{-1}\Con{z}{y}^{-1})\Con{y}{z}\Con{z}{y}\Mul{x^{-\epsilon}}{y}\Con{y}{z}^{-1}\Con{y}{x}^{\epsilon}&&\\
&=\Con{y}{x}^{-\epsilon}\Mul{x^{-\epsilon}}{y}^{-1}\Con{y}{z}\Con{z}{y}\Mul{x^{-\epsilon}}{y}\Con{y}{z}^{-1}\Con{y}{x}^{\epsilon}&&\text{(relation R1.2)}\\
&=\Con{y}{x}^{-\epsilon}\Con{y}{z}(\Con{y}{z}^{-1}\Mul{x^{-\epsilon}}{y}^{-1}\Con{y}{z})(\Con{z}{y}\Mul{x^{-\epsilon}}{y})\Con{y}{z}^{-1}\Con{y}{x}^{\epsilon}&&\\
&=\Con{y}{x}^{-\epsilon}\Con{y}{z}(\Con{z}{y}\Mul{x^{-\epsilon}}{y})(\Con{y}{z}^{-1}\Mul{x^{-\epsilon}}{y}^{-1}\Con{y}{z})\Con{y}{z}^{-1}\Con{y}{x}^{\epsilon}&&\text{(relation C2.1)}\\
&=\Con{y}{x}^{-\epsilon}\Con{y}{z}\Con{z}{y}\big[\Mul{x^{-\epsilon}}{y},\Con{y}{z}^{-1}\big]\Con{y}{z}^{-1}\Con{y}{x}^{\epsilon}&&\\
&=f(g,s)&&\text{(computation in $F(S_K)$).}\\
\end{align*}
This completes the proof of Proposition~\ref{proposition:action}.
\end{proof}

\appendix
\section{Appendix : Simple closed curves and conjugacy classes in $\pi_1(\Sigma_g)$}
\label{appendix:curvesdesc}

In this appendix, we prove Proposition~\ref{proposition:curvesdesc} from \S \ref{section:introduction}.  Recall
that it asserts that a collection $c_1,\ldots,c_k$ of conjugacy classes in $\pi_1(\Sigma_g)$ corresponds
to a collection $\gamma_1,\ldots,\gamma_k$ of homotopy classes of disjoint oriented non-nullhomotopic
simple closed curves on $\Sigma_g$ such that $\Sigma_g \setminus (\gamma_1 \cup \cdots \cup \gamma_k)$
is connected if and only if there exists a standard basis $a_1,b_1,\ldots,a_g,b_g$ for $\pi_1(\Sigma_g)$
with $c_i = \Conj{a_i}$ for $1 \leq i \leq k$.

Define
\[\mathcal{B} = \Set{$(\Conj{a_1},\Conj{b_1},\ldots,\Conj{a_g},\Conj{b_g})$}{$a_1,b_1,\ldots,a_g,b_g$ is a standard basis for $\pi_1(\Sigma_g)$}\]
and
\[\mathcal{X}_k = \Set{$(\gamma_1,\ldots,\gamma_k$)}{$\gamma_1,\ldots,\gamma_k$ are homotopy classes of oriented closed curves on $\Sigma_g$}.\]
Let $\mathcal{Y}_k \subset \mathcal{X}_k$ be
\begin{align*}
\mathcal{Y}_k 
=\{\text{$(\gamma_1,\ldots,\gamma_k) \in \mathcal{X}_k$ $|$ }&\text{$\gamma_1,\ldots,\gamma_k$ are distinct, simple and disjointly realizable} \\
&\text{with $\Sigma_g \setminus (\gamma_1 \cup \cdots \cup \gamma_k)$ is connected}\}.
\end{align*}
\begin{remark}
The fact that $\Sigma_g \setminus (\gamma_1 \cup \cdots \cup \gamma_k)$ is connected implies that $\gamma_1,\ldots,\gamma_k$ are non-nullhomotopic.
\end{remark}
There is a natural map $\rho : \mathcal{B} \rightarrow \mathcal{X}_k$ taking $(\Conj{a_1},\Conj{b_1},\ldots,\Conj{a_g},\Conj{b_g}) \in \mathcal{B}$
to the tuple of closed curves corresponding to $(\Conj{a_1},\Conj{a_2},\ldots,\Conj{a_k})$, and our goal is to show
that $\rho(\mathcal{B}) = \mathcal{Y}_k$.  

The group $\Out(\pi_1(\Sigma_g))$ acts on $\mathcal{B}$ and the group $\Mod(\Sigma_g)$ acts on $\mathcal{X}_k$.  The 
Dehn-Nielsen-Baer theorem (see \cite{FarbMargalitPrimer}) says that 
$\Mod(\Sigma_g) \cong \Out(\pi_1(\Sigma_g))$, and under this identification the map $\rho$ is equivariant.  Now, 
it is clear that $\Out(\pi_1(\Sigma_g))$ acts transitively on $\mathcal{B}$.  Moreover, the action of $\Mod(\Sigma_g)$ on
$\mathcal{X}_k$ preserves $\mathcal{Y}_k$, and using the ``change of coordinates'' principle from \cite[\S 1.3]{FarbMargalitPrimer} one
can show that $\Mod(\Sigma_g)$ acts transitively on $\mathcal{Y}_k$.  Since the image of $\rho$ certainly contains at least
one element of $\mathcal{Y}_k$, it follows that the image equals $\mathcal{Y}_k$, as desired.

\noindent
\begin{tabular*}{\linewidth}[t]{@{}p{\widthof{Department of Mathematics 253-37}+1in}@{}p{\linewidth - \widthof{Department of Mathematics 253-37}-1in}@{}}
{\raggedright
Matthew Day\\
Department of Mathematics 253-37\\
California Institute of Technology\\
Pasadena, CA 91125\\
E-mail: {\tt mattday@caltech.edu}}
&
{\raggedright
Andrew Putman\\
Department of Mathematics\\
Rice University, MS 136 \\
6100 Main St.\\
Houston, TX 77005\\ 
E-mail: {\tt andyp@rice.edu}}
\end{tabular*}

\end{document}